\documentclass[11pt]{article}

\usepackage[latin1]{inputenc}
\usepackage{t1enc}
\usepackage{amsmath, amssymb, amsfonts, amsthm, amsopn,epsfig}
\usepackage[none]{hyphenat}
\usepackage{float}
\usepackage{graphicx}
\usepackage{caption}
\usepackage[toc,page]{appendix}
\usepackage{epstopdf}
\usepackage{etoolbox}
\usepackage[colorlinks=true]{hyperref}
\usepackage{cleveref}
\usepackage{tikz}
\usetikzlibrary{
  arrows.meta,
  calc,
  positioning,
  fit,
  backgrounds,
  decorations.pathreplacing,
  shapes.geometric
}
\usepackage{pgfplots}
\pgfplotsset{compat=1.18}
\definecolor{surfaceblue}{RGB}{45,100,175}
\definecolor{cellone}{RGB}{239,166,71}
\definecolor{celltwo}{RGB}{112,173,103}
\definecolor{pointred}{RGB}{190,45,45}

\theoremstyle{plain}
\newtheorem{theorem}{Theorem}[section]

\newtheorem{corollary}[theorem]{Corollary}
\newtheorem{claim}[theorem]{Claim}
\newtheorem{lemma}[theorem]{Lemma}
\newtheorem{conjecture}[theorem]{Conjecture}
\newtheorem{problem}[theorem]{Problem}

\AddToHook{env/lemma/begin}{\crefalias{theorem}{lemma}}
\AddToHook{env/claim/begin}{\crefalias{theorem}{claim}}
\AddToHook{env/proposition/begin}{\crefalias{theorem}{proposition}}
\AddToHook{env/corollary/begin}{\crefalias{theorem}{corollary}}
\AddToHook{env/problem/begin}{\crefalias{theorem}{problem}}

\crefname{theorem}{theorem}{theorems}
\Crefname{theorem}{Theorem}{Theorems}

\crefname{lemma}{lemma}{lemmas}
\Crefname{lemma}{Lemma}{Lemmas}

\crefname{claim}{claim}{claims}
\Crefname{claim}{Claim}{Claims}

\crefname{proposition}{proposition}{propositions}
\Crefname{proposition}{Proposition}{Propositions}

\crefname{corollary}{corollary}{corollaries}
\Crefname{corollary}{Corollary}{Corollaries}

\crefname{problem}{problem}{problems}
\Crefname{problem}{Problem}{Problems}

\usepackage{dsfont}
\hypersetup{
	colorlinks=true,
	linkcolor=blue,
	filecolor=magenta,
	urlcolor=cyan,
	citecolor=blue
}
\usepackage[top=22mm, bottom=22mm, left=22mm, right=22mm]{geometry}
\usepackage{comment}

\theoremstyle{definition}
\newtheorem{definition}{Definition}

\DeclareMathOperator{\rank}{rank}
\DeclareMathOperator{\tr}{tr}

\newcommand{\eps}{\varepsilon}

\DeclareMathOperator{\disc}{disc}
\DeclareMathOperator{\herdisc}{herdisc}

\title{Discrepancy of geometric incidences}
\author{Azem Adibelli\thanks{Ume\r{a} University, \emph{e-mail}: \textbf{\{azem.adibelli,istvan.tomon\}@umu.se}. Research supported in part by the Swedish Research Council grant VR 2023-03375.}
	\and 
	Istv\'an Tomon\footnotemark[1]}
\date{}

\begin{document}
	\sloppy 
	\maketitle
	
	\begin{abstract}
		We study the combinatorial (red-blue) discrepancy of finite point sets with respect to hyperplanes and, more generally, bounded-complexity affine algebraic sets. We prove that every $n$-point set in a real Euclidean space admits a red-blue coloring for which every affine algebraic set of dimension at most $D$ and degree at most $k$ has discrepancy at most
$n^{\frac12-\frac{1}{2(D+1)}-\eps}$
for some $\eps=\eps(D,k)>0$. This gives a polynomial improvement over the straightforward VC-dimension bound $\tilde O(n^{\frac12-\frac{1}{2(D+1)}})$. In the opposite direction, we construct $n$-point sets in $\mathbb R^d$ whose discrepancy with respect to hyperplanes is
$\tilde\Omega(n^{\frac12-\frac{1}{d+1}}),$
extending the point-line discrepancy lower bound of Chazelle and Lvov. 
		
		We present further applications of our methods in communication complexity, concerning separation between randomized communication cost and deterministic communication cost with access to equality oracle.
	\end{abstract}

	\section{Introduction}
	
	A central problem in combinatorial discrepancy theory is to determine how evenly a finite set can be partitioned with respect to a prescribed family of subsets. The subject has its roots in the classical theory of irregularities of distribution, developed in, for example, \cite{BC,DT}, but has since grown into a broad and active area with connections to computer science, numerical analysis, optimization, and several other fields. General introductions to the theory can be found in \cite{BN,Ch,MatBook}.
	
	Let $\mathcal{F}\subset 2^X$ be a set system over a finite ground set $X$. The \emph{combinatorial discrepancy} of $\mathcal{F}$ (which we will simply refer to as discrepancy) is defined as follows. Given a red-blue coloring of $X$, the discrepancy of the coloring with respect to $\mathcal{F}$ is the maximum difference of the numbers of red and blue points over all members of $\mathcal{F}$. The discrepancy of $\mathcal{F}$ is then the minimum discrepancy over all red-blue colorings. Formally,
	$$\disc(\mathcal{F})=\min_{\chi:X\rightarrow \{-1,1\}}\max_{A\in \mathcal{F}} \left| \sum_{x\in A}\chi(x)\right|.$$
	In the case of geometric discrepancy problems, it is also common to use the notation
	$$\disc(X,\mathcal{R}):=\disc(\{R\cap X: R\in \mathcal{R}\}),$$
	where $\mathcal{R}$ is a family of \emph{ranges} (typically a collection of geometric objects, such as half-spaces, balls, axis-parallel boxes, etc.) and $X$ is a set of points in a Euclidean space.
	
	A classical result in discrepancy theory is due to Spencer \cite{Spencer}, who showed that a family $\mathcal F$ of size $n$ on an $n$-element ground set satisfies $\disc(\mathcal F)=O(\sqrt{n})$, and more generally for $|\mathcal F|\geq n$, we have $$\disc(\mathcal F)\ll \sqrt{n\log \frac{2|\mathcal{F}|}{n}}.$$ While these bounds are optimal in general, they can be substantially improved for geometric families. 
	
	A classical problem in geometric discrepancy theory is concerned with the family $\mathcal{HS}_d$ of half-spaces in $\mathbb{R}^d$. After a long line of work, it was proved by Matou\v{s}ek \cite{MatVCdim} that an $n$-element point set $X\subset \mathbb{R}^d$ satisfies $$\disc(X,\mathcal{HS}_d)=O(n^{\frac{1}{2}-\frac{1}{2d}}),$$ and a matching lower bound was later established by Alexander \cite{Alexander}. Further proofs were later provided by Chazelle, Matou\v{s}ek and Sharir \cite{CMS}. Another classical problem, called \emph{Tusn\'ady's problem}, is concerned with the discrepancy of  $\mathcal{B}_d$, denoting the family of axis-parallel boxes in $\mathbb{R}^d$. This problem is also closely related to the \emph{''Great open problem''} in analytic discrepancy theory, see \cite{MNT} for further discussion. The currently best known bounds are $$\Omega((\log n)^{d-1})=\max_{|X|=n}\disc(X,\mathcal B_d)=O((\log n)^{d-1/2}),$$ the lower bound  due to Matou\v{s}ek, Nikolov, and Talwar \cite{MNT}, and the upper bound due to Nikolov \cite{Nikolov}.
	
	However, the result most directly relevant to the present paper is due to Chazelle and Lvov \cite{ChL}, which states that for the family $\mathcal{L}$ of lines in $\mathbb{R}^2$, we have
	$$\max_{|X|=n}\disc(X,\mathcal{L})=\tilde{\Theta}(n^{1/6}).$$
	Here and throughout the paper, the notation $\widetilde{O}$, $\widetilde{\Omega}$, and $\widetilde{\Theta}$ suppresses polylogarithmic factors. This result naturally lies at the interface of discrepancy theory and incidence geometry. It is therefore natural to ask about discrepancy in other geometric incidence problems, such as incidences between points and hyperplanes in $\mathbb R^d$, or more generally,  bounded-complexity affine algebraic sets, such as spheres, cones, and cylinders. This motivates the following problem, which is the central question of our manuscript.
	
	\begin{problem}\label{problem:main}
		Let $\mathcal{H}_d$ denote the family of hyperplanes in $\mathbb{R}^d$, and let $\mathcal{A}_{D,k}$ denote the family of affine algebraic sets of dimension at most $D$ and degree at most $k$ in $\mathbb{R}^d$. Determine the order of magnitude of 
		$$\max_{X\subset \mathbb{R}^d, |X|=n}\disc(X,\mathcal{H}_d)\quad\text{and}\quad\max_{X\subset \mathbb{R}^d, |X|=n}\disc(X,\mathcal{A}_{D,k}).$$
	\end{problem}
	
	The arguments of Chazelle and Lvov\cite{ChL} establishing the nearly matching bounds for lines in the plane rely crucially on planar geometry and do not extend to higher dimensions. Consequently, even the case of hyperplanes presents substantial new difficulties, making the higher-dimensional problem both challenging and particularly compelling. 
    
    We remark that we choose not to denote the ambient dimension $d$ for the family $\mathcal{A}_{D,k}$ deliberately, as a standard projection argument (see \Cref{lemma:projection}) shows that we can always assume $d=D+1$ without loss of generality. 
    
    In the next sections, we discuss lower and upper bounds for \Cref{problem:main} separately and present our main results.
	\medskip
	
	\noindent
	\textbf{Incidence geometry.} \Cref{problem:main} is strongly motivated by classical problems in geometric incidence theory \cite{FPSSZ,GK,ST,SzT,TY}. A core paradigm in the area is that a large number of incidences between points in a Euclidean space and bounded-complexity algebraic sets are typically caused by large degenerate subconfigurations.  For example, a large number of incidences between points and hyperplanes forces many points to lie on a 2-codimensional affine subspace that is contained in many hyperplanes \cite{AS,MST}. From the perspective of discrepancy, such structured subconfigurations may be grouped and colored in a balanced manner. Therefore, in a certain sense, $\disc(X,\mathcal{A}_{D,k})$ measures how efficiently can one assemble the incidence structure from lower complexity subconfigurations.

	\subsection{Lower bounds}
	
	The lower bound for the discrepancy of lines $\disc(X,\mathcal{L})=\Omega(n^{1/6})$ essentially follows from two properties of point-line incidences: (i) there exist $n$ points and $n$ lines with $\Omega(n^{4/3})$ incidences \cite{Erdos}, and (ii) point-line incidence graphs are four-cycle-free. Indeed, in order to prove this lower bound, Chazelle and Lvov \cite{ChL} introduced the trace-bound, which is applicable to families $\mathcal{F}\subset 2^X$ that are $C_4$-free or contain few copies of $C_4$ (here, $C_4$ denotes the four-cycle, which corresponds to a $4$-tuple $(x,y,A,B)\in X^2\times \mathcal F^2$ such that $x\neq y, A\neq B$ and $x,y\in A\cap B$). We discuss this approach in detail in \Cref{sect:trace-bound}. We present a strengthening of the trace-bound which can be used to prove the lower bound $\disc(X,\mathcal{H}_3)=\tilde{\Omega}(n^{1/4})$ in dimension 3, however, we argue that similar approaches completely fail in higher dimensions; see \Cref{sect:dim3}.
	
	Instead, we use analytic and linear algebraic tools to bound the discrepancy of a certain grid-like configuration of points. Our main lower bound is presented in the following theorem.
	
	\begin{theorem}\label{thm:main_lower}
		Let $\mathcal{H}_d$ denote the family of hyperplanes in $\mathbb{R}^d$. There exists a set $X\subset \mathbb{R}^d$ of size $n$ such that
		$$\disc(X,\mathcal{H}_d)=\Omega_d\left(n^{\frac{1}{2}-\frac{1}{d+1}}\cdot (\log n)^{-1}\right).$$
	\end{theorem}
	
	\noindent
	The proof of \Cref{thm:main_lower} is presented in \Cref{sect:main_lower}. We also strongly believe that this lower bound is sharp, however, we are only able to prove weaker upper bounds for $d\geq 3$. We leave it as an intriguing open problem to close the gap.

    \begin{conjecture}
        Let $\mathcal{H}_d$ denote the family of hyperplanes in $\mathbb{R}^d$. Then for every set $X\subset \mathbb{R}^d$ of size $n$:
		$$\disc(X,\mathcal{H}_d)=\tilde{O}\left(n^{\frac{1}{2}-\frac{1}{d+1}}\right).$$
    \end{conjecture}
	
	\subsection{Upper bounds}
	
	The upper bound on the discrepancy of lines $\disc(X,\mathcal{L})=\tilde{O}(n^{1/6})$ follows simply from the sparsity of point-line incidences. The Szemer\'edi-Trotter theorem \cite{SzT} implies that an incidence graph of $n$ points with respect to a set of lines is $r:=O(n^{1/3})$-degenerate, which implies the upper bound $\tilde{O}(r^{1/2})$ on the discrepancy; see \Cref{sect:prelim} for more details. Unfortunately, this argument already fails if one considers incidence graphs of points and hyperplanes in $\mathbb{R}^d$ for $d\geq 3$. Indeed, if $X$ is a set of $n$ points on a line, and $\mathcal{H}$ is family of $n$ hyperplanes containing this line, the incidence graph has degeneracy $n$.
	
	On the other hand, an immediate upper bound for our question can be derived from VC-dimension arguments. Given a range space $\mathcal{R}$, its \emph{primal shatter function} is $\pi_{\mathcal{R}}(m)$, denoting the maximum size of the projection $\mathcal{R}|_X=\{R\cap X:R\in\mathcal{R}\}$ over all sets $X$ of size $m$. Moreover, the \emph{dual shatter function} of $\mathcal{R}$ is $\pi^*_{\mathcal{R}}(m)$, denoting the primal shatter function of the \emph{dual family}. Here, writing $R_x=\{R\in \mathcal{R}:x\in R\}$, the dual family is defined as $\mathcal{R}^*=\{R_x:x\in \bigcup \mathcal{R}\}$.
	
	It was proved by Matou\v{s}ek \cite{MatVCdim} that if the primal shatter function of $\mathcal{R}$ satisfies $\pi_{\mathcal{R}}(m)=O(m^d)$, then for every $n$-element set $X$, $\disc(X,\mathcal{R})=O(n^{\frac{1}{2}-\frac{1}{2d}})$. This bound is sharp, as witnessed by the family of half-spaces in dimension $d$. Moreover, Matou\v{s}ek, Welzl and Wernisch \cite{MWW} proved that if the dual shatter function satisfies $\pi_{\mathcal{R}}^*(m)=O(m^d)$, then $\disc(X,\mathcal{R})=O(n^{\frac{1}{2}-\frac{1}{2d}}\sqrt{\log n})$. This bound is also sharp, as proved by Alon, R\'onyai and Szab\'o \cite{ARSz},  extending the $d=2,3$ cases previously proved by Matou\v{s}ek \cite{MatdualVC}.
	
	The primal and dual shatter functions of the family $\mathcal{H}_d$ of hyperplanes in $\mathbb{R}^d$ are both $O(m^d)$. More generally, in case $\mathcal{A}_{D,k}$ is a family of affine algebraic sets of dimension at most $D$ and degree at most $k$ (in any real space), the dual shatter function of $\mathcal{A}_{D,k}$ is $(O_{D,k}(m))^{D+1}$. This follows from the Milnor-Thom theorem \cite{Milnor,Thom}, which states that $m$ polynomials $f_1,\dots,f_m\subset \mathbb{R}[x_1,\dots,x_d]$ of degree at most $k$ determine $O(km/d)^{d}$ different sign patterns $(f_1(x),\dots,f_m(x))$, $x\in \mathbb{R}^d$ (we omit further details, as these claims will not be used later). Hence, we get 
	$$\disc(X,\mathcal{A}_{D,k})=\tilde{O}(n^{\frac{1}{2}-\frac{1}{2(D+1)}})$$
	for any $n$-element set $X$.  Therefore, we view the bound $\tilde{O}(n^{\frac{1}{2}-\frac{1}{2(D+1)}})$ as the baseline which we want to improve. It is not at all clear why one should expect an improvement, as this bound is sharp in two closely related settings: (i)  $\mathcal{R}$ is the family of half-spaces in $\mathbb{R}^d$, (ii) $\mathcal{R}$ is the family of hyperplanes over a $d$-dimensional finite field. In other words, (i) shows that the VC-dimension bound cannot be improved even if the family is semialgebraic of constant description complexity, while (ii) shows that the algebraic nature of hyperplanes is also not enough. However, we can prove that for real algebraic sets, one can achieve an improvement. The relevant algebraic geometry background is discussed in \Cref{sect:alg_geo}.
	
	\begin{theorem}\label{thm:main_upper}
		Let $D$ and $k$ be positive integers, then there exists $\eps=\Omega(D^{-2}\binom{D+1+k}{k}^{-1})$ such that the following holds. Let $\mathcal{A}_{D,k}$ be the family of affine algebraic sets of dimension at most $D$ and degree at most $k$ in $\mathbb{R}^d$, and let $X\subset \mathbb{R}^d$ be a set of size $n$. Then
		$$\disc(X,\mathcal{A}_{D,k})\leq O_{D,k}\left(n^{\frac{1}{2}-\frac{1}{2(D+1)}-\eps}\right).$$
	\end{theorem}
	
	We present the proof of \Cref{thm:main_upper} in \Cref{sect:upper}. We highlight the following immediate corollary of this result.
	
	\begin{corollary}
		Let $d$ and $k$ be positive integers, then there exists $\eps=\eps(d,k)=\Omega(d^{-2}\binom{d+k}{k}^{-1})$ such that the following holds. Let $\mathcal{Z}_{d,k}$ be the family of zero sets of $d$-variate polynomials of degree at most $k$, and let $X\subset \mathbb{R}^d$ be a set of size $n$. Then
		$$\disc(X,\mathcal{Z}_{d,k})\leq O_{d,k}\left(n^{\frac{1}{2}-\frac{1}{2d}-\eps}\right).$$
	\end{corollary}
	
	\noindent
	The case of hyperplanes corresponds to the $k=1$ subcase of this corollary. Therefore, 
	$$\disc(X,\mathcal{H}_d)\leq O_{d}\left(n^{\frac{1}{2}-\frac{1}{2d}-\Omega(\frac{1}{d^3})}\right).$$
	In particular, we get the following exact exponents in small dimensions. Writing $\disc(X,\mathcal{H}_d)=\tilde{\Theta}(n^{\alpha_d})$ (assuming such an exponent $\alpha_d$ exists), we prove that 
	$$\alpha_2=\frac{1}{6},\ \alpha_3\in \left[\frac{1}{4},\frac{2}{7}\right],\text{ and }\alpha_4\in \left[0.3,0.355\right].$$ 
    We present a simplified proof of \Cref{thm:main_upper} in the case of hyperplanes in \Cref{sect:hyperplanes}.

	\subsection{Communication complexity} 
    
    Our methods have further implications in communication complexity. Building on recent ideas of Goh and Hatami \cite{GH_totally_real}, for every $c>0$, we present a family of $n\times n$ Boolean matrices $M$ whose randomized communication cost is $O_c(1)$, while their $\gamma_2$-norm satisfies $\gamma_2(M)\geq n^{1/2-c}$. This achieves extreme separation between the randomized communication complexity and the deterministic communication complexity with access to equality oracle. In particular, we extend the point-line incidence matrix construction of \cite{GH_totally_real}, and show that high-dimensional point-hyperplane incidence matrices suffice. We discuss the relevant background and present our results in \Cref{sect:comm_compl}.
	
	\section{Preliminaries}\label{sect:prelim}
	
	In this section, we present our notation and discuss basic results and tools used in our proofs.
	
	\subsection{Discrepancy}
	
	Let $\mathcal{F}\subset 2^X$ be a set system. Instead of the discrepancy, it is often more convenient to consider the \emph{hereditary discrepancy} of $\mathcal{F}$, which is the maximum discrepancy over all projections of $\mathcal{F}$. Formally, given $Y\subset X$ and writing $\mathcal{F}|_{Y}=\{A\cap Y:A\in \mathcal{F}\}$ for the restriction of $\mathcal F$, we define
	$$\herdisc(\mathcal{F})=\max_{Y\subset X}\disc(\mathcal{F}|_Y).$$
	In case one is interested in the extremal behavior of the discrepancy of certain range spaces, it makes very little difference whether one considers the discrepancy or the hereditary discrepancy. However, as we shall see, the latter is easier to estimate with algebraic techniques.
	
	Consider the discrepancy from a matrix theoretic perspective. The \emph{incidence matrix} of a family of sets $\mathcal{F}\subset 2^X$ is the Boolean matrix $M$, whose columns are indexed by the elements of $X$, rows are indexed by the members of $\mathcal{F}$, and 
	$$M(A,x)=\begin{cases}1&\text{if }x\in A,\\
		0&\text{otherwise.}\end{cases}$$
	Then we can also write
	$$\disc(\mathcal{F})=\min_{\chi\in \{-1,1\}^X} ||M\chi||_{\infty}.$$
	In particular, we can define the discrepancy of a (not necessarily Boolean) $m\times n$ matrix $M$ as
	$$\disc(M)=\min_{\chi\in \{-1,1\}^n} ||M\chi||_{\infty},$$
	and the hereditary discrepancy of $M$ as
	$$\herdisc(M)=\max_{N}\disc(N),$$
	where the maximum is taken over all submatrices $N$ of $M$. One of the main advantages of the hereditary discrepancy is that it can be closely approximated by a certain matrix norm, called the $\gamma_2$-norm.
	
	\subsection{Factorization norms}
	
	Given a matrix $M\in \mathbb{C}^{m\times n}$, its \emph{$\gamma_2$-norm} is defined as
	$$\gamma_2(M)=\min_{M=UV} ||U||_{\text{row}}||V||_{\text{col}},$$
	where $||U||_{\text{row}}$ is the maximal $\ell_2$-norm of a row vector of $U$, while $||V||_{\text{col}}$ is the maximal $\ell_2$-norm of a column of $V$. The $\gamma_2$-norm can also be defined as the following maximization program:
	$$\gamma_2(M)=\max_{\substack{u\in \mathbb{C}^m, v\in \mathbb{C}^n\\ ||u||_2=||v||_2=1}} ||M\circ uv^*||_{\text{tr}},$$
	where $||N||_{\text{tr}}$ is the \emph{trace-norm} of $N$, that is, the sum of singular values of $N$, and $\circ$ is the Hadamard product (entry-wise product). 
	
	Our interest in the $\gamma_2$-norm stems from the following extraordinary result of Matou\v{s}ek, Nikolov, and Talwar \cite{MNT}, showing that it approximates the hereditary discrepancy up to logarithmic factors. 
	
	\begin{theorem}\label{lemma:gamma_herdisc}
		Let $M\in\mathbb{R}^{m\times n}$. Then
		$$\Omega\left(\frac{\gamma_2(M)}{\log m}\right)\leq \herdisc(M)\leq  O(\gamma_2(M)\sqrt{\log m}).$$
	\end{theorem}
	
	\noindent
	\textbf{Remark.} In case $M$ is the incidence matrix of a family of sets $\mathcal{F}$ of bounded VC-dimension, which is true for all families studied in this paper, the Sauer-Shelah lemma ensures that $m\leq n^{O(1)}$. Thus, the $\log m$ factor in the previous theorem can be safely replaced by $\log n$, which we will do without further remarks.
	
	\bigskip
	
	In what follows, we collect a number of simple properties of the factorization norm and the trace-norm that will be used in our proofs. Each of these claims can be found in \cite{BHT}, for example.
	
	\begin{enumerate}
		\item \textbf{(norm)} $\gamma_2(.)$ is indeed a matrix norm, so $$\gamma_2(cM)=|c|\gamma_2(M)\quad\text{and}\quad\gamma_2(M+N)\leq \gamma_2(M)+\gamma_2(N);$$
		\item \textbf{(normalized trace-norm)} if $M\in \mathbb{C}^{m\times n}$, then $\gamma_2(M)\geq \frac{1}{\sqrt{mn}}||M||_{\tr}$;
		\item \textbf{(tensor product)} The \emph{tensor product} of $M\in \mathbb{C}^{m\times n}$ and $N\in \mathbb{C}^{m'\times n'}$ is the matrix $M\otimes N\in \mathbb{C}^{mm'\times nn'}$ defined as $(M\otimes N)((i,i'),(j,j'))=M(i,j)N(i',j')$ for $i\in [m], i'\in [m'], j\in [n], j'\in [n']$. We have 
        $$\gamma_2(M\otimes N)=\gamma_2(M)\gamma_2(N)\quad\text{and}\quad||M\otimes N||_{\tr}=||M||_{\tr}||N||_{\tr}.$$
		\item \textbf{(direct sum)} The \emph{direct sum} of $M\in \mathbb{C}^{m\times n}$ and $N\in \mathbb{C}^{m'\times n'}$ is the matrix $M\oplus N\in \mathbb{C}^{(m+m')\times (n+n')}$ we get by putting a copy of $M$ in the top-left corner, a copy of $N$ in the bottom-right corner, and setting the rest of the entries to 0. We have  $$\gamma_2(M\oplus N)=\max(\gamma_2(M),\gamma_2(N))\quad\text{and}\quad||M\oplus N||_{\tr}=||M||_{\tr}+||N||_{\tr}.$$
		\item \textbf{(blow-up)} A matrix $N$ is a \emph{blow-up} of a matrix $M$, if we can get $N$ by repeating rows and columns of $M$. If $N$ is a blow-up of $M$, then $\gamma_2(N)=\gamma_2(M)$.
		\item \textbf{(degeneracy)} If $M$ is a Boolean matrix such that every row has at most $d$ one-entries, or every column has at most $d$ one-entries, then $\gamma_2(M)\leq \sqrt{d}$. More generally, a Boolean matrix $M$ is \emph{$d$-degenerate} if every submatrix of $M$ contains a row or a column with at most $d$ one-entries. The \emph{degeneracy} of $M$ is the smallest $d$ for which $M$ is $d$-degenerate. If $M$ is $d$-degenerate, then $\gamma_2(M)\leq 2\sqrt{d}$.
		
		\item \textbf{(trace-norm)} Let $\langle A,M\rangle=\tr(A^*M)$ denote the usual dot-product between matrices, and let $||A||_{\text{op}}=\max_{||x||_2=1}||Ax||_2$ denote the operator norm. Then $$||M||_{\text{tr}}=\max_{\substack{A\in \mathbb{C}^{m\times n}\\ ||A||_{\text{op}}\leq 1}}|\langle A,M\rangle|.$$
	\end{enumerate}
	
	The following is a somewhat more involved inequality concerning the $\gamma_2$-norm, which is about the concatenation of sparse matrices.
	
	\begin{lemma}\label{lemma:concatenation}
		For $i=1,\dots,q$, let $M_i$ be an $m\times n_i$ sized matrix such that $\gamma_2(M_i)\leq \gamma$.  Assume that for every $j\in [m]$, there are at most $t$ indices $i\in [q]$ such that the $j$-th row of $M_i$ is not all-zero. Let $M$ be the $m\times (n_1+\dots+n_q)$ matrix we get by concatenating $M_1,\dots,M_q$ horizontally, i.e. $M=\begin{pmatrix}
			M_1 & \dots & M_q
		\end{pmatrix}$. Then
		$$\gamma_2(M)\leq \gamma\sqrt{t}.$$
	\end{lemma}
	
	\begin{proof}
		For $i\in [q]$, write $M_i=U_iV_i$ such that $||U_i||_{\text{row}}\leq 1$ and $||V_i||_{\text{col}}\leq \gamma$. In case the $j$-th row of $M_i$ is 0, we set the $j$-th row of $U_i$ to be 0 as well. Let the size of $U_i$ be $m\times r_i$, then the size of $V_i$ is $r_i\times n_i$. Let $U$ be the matrix we get by concatenating $U_1,\dots,U_q$ horizontally. Let $V$ be the $(r_1+\dots+r_q)\times n$ matrix, where for $i\in [q]$, the submatrix on the indices $$\{r_1+\dots+r_{i-1}+1,\dots,r_1+\dots+r_i\}\times \{n_1+\dots+n_{i-1}+1,\dots,n_1+\dots+n_i\}$$ is equal to $V_i$, and all other unspecified entries are 0. Then $M=UV$. For example, in case $q=3$, we have the following factorization:
		$$\begin{pmatrix} M_1 & M_2 & M_3 \end{pmatrix}=\begin{pmatrix} U_1 & U_2 & U_3\end{pmatrix}\begin{pmatrix}
			V_1 & 0 & 0\\ 0 & V_2 & 0 \\ 0 & 0 & V_3
		\end{pmatrix}.$$
		
		Up to deleting 0 entries, each column of $V$ is a column of some $V_i$, so we have $||V||_{\text{col}}\leq \gamma$. On the other hand, if $j\in [m]$ and $u_j$ is $j$-th row vector of $U$, and $u_{i,j}$ is the $j$-th row vector of $U_i$, then
		$$||u_j||_2^2=\sum_{i=1}^q ||u_{i,j}||_2^2\leq t.$$
		Here, we used that $||u_{i,j}||\leq 1$ and there are at most $t$ indices $i$ such that $u_{i,j}\neq 0$. Thus, $||U||_{\text{row}}\leq \sqrt{t}$. In conclusion, as $M=UV$, we have $$\gamma_2(M)\leq ||U||_{\text{row}}||V||_{\text{col}}\leq \gamma\sqrt{t}.$$
	\end{proof}
	
	\section{Lower bounds}
	
	In this section, we discuss different approaches to prove lower bounds for the discrepancy and the $\gamma_2$-norm, and present the proof of \Cref{thm:main_lower}.
	
	\subsection{Trace bound}\label{sect:trace-bound}
	
	In order to prove lower bounds for the discrepancy of point-line incidences, Chazelle and Lvov \cite{ChL} introduced the \emph{trace bound}. Given a matrix $M$, let $||M||_q$ denote the Schatten $q$-norm of $M$, that is, 
	$$||M||_q=\left(\sum_{i=1}^s \sigma_i^q\right)^{1/q},$$
	where $\sigma_1,\dots,\sigma_{s}$ are the singular values of $M$. The Schatten 2-norm and 4-norm are of particular interest, as they can be calculated as follows $$||M||_{2}^2=\tr(M^*M)\quad\text{and}\quad ||M||_4^4=\tr((M^*M)^2).$$
	Combinatorially, if $M$ is a Boolean matrix, then $||M||_2^2$ is the number of 1-entries of $M$, and $||M||_4^4$ is the number of homomorphisms of the four-cycle $C_4$ in the corresponding bipartite graph of $M$. Here, a $C_4$ refers to a $2\times 2$ all-ones submatrix, while a homomorphism of a $C_4$ is an all ones submatrix $M[\{x_1,x_2\}\times \{y_1,y_2\}]$, where $x_1$ and $x_2$ are not necessarily different, and similarly for $y_1$ and $y_2$. Let $M$ be an $m\times n$ Boolean matrix, and let $t=||M|||_2^2$ and $s=||M||_4^4$. Then the trace bound \cite{ChL} states that
	\begin{equation}\label{equ:trace_bound}
		\herdisc(M)\geq \frac{1}{4}c^{ns/t^2}\sqrt{\frac{t}{n}},
	\end{equation}
	where $c\in (0,1)$ is an absolute constant. This bound is only meaningful if $ns/t^2\ll \log n$, which happens exactly when $M$ contains no, or very few copies of $C_4$. However, by considering the $\gamma_2$-norm, one can establish a lower bound in a similar spirit that is also applicable for moderate $C_4$ counts. The same inequality is also established in \cite{CHHNPS}, but we present its short proof for completeness.
	
	\begin{lemma}\label{lemma:trace_bound}
		Let $M\in \mathbb{C}^{m\times n}$, then
		$$||M||_{\tr}\geq ||M||_2^{3}\cdot ||M||_4^{-2}\quad\text{and}\quad\gamma_2(M)\geq \frac{1}{\sqrt{mn}}||M||_2^{3} ||M||_4^{-2}.$$
	\end{lemma}
	
	\begin{proof}
		For a vector $x\in \mathbb{C}^n$, write $||x||_p=(|x(1)|^p+\dots+|x(n)|^p)^{1/p}$ for the $p$-th quasi-norm of $x$. Let $\sigma=(\sigma_1,\dots,\sigma_n)$ be the vector of singular values of $M$. Then
		$$\gamma_2(M)\geq \frac{1}{\sqrt{mn}}||M||_{\tr}=\frac{1}{\sqrt{mn}}||\sigma||_1.$$
		Moreover, we have $||M||_q=||\sigma||_q$. In order to bound the right hand side, we use the following generalization of H\"older's inequality.
		
		\begin{lemma}[Generalized H\"older's inequality]
			Let $x$ be a vector. If $p_1,\dots,p_k,r>0$ such that $\frac{1}{r}=\frac{1}{p_1}+\dots+\frac{1}{p_k}$, then $||x||_{p_1}\dots ||x||_{p_k}\geq ||x^{k}||_r $, where $x^{k}$ is the vector defined as $(x^k)(i)=(x(i))^k$.
		\end{lemma}
		
		Applying H\"older's inequality  with the parameters $k=3,p_1=1,p_2=p_3=4,r=2/3$,  we arrive at the inequality
		\begin{equation*}
			||\sigma||_1 ||\sigma||_4^2 \geq ||\sigma^3||_{2/3}.
		\end{equation*}
		Here, $||\sigma^3||_{2/3}=||\sigma||_2^3$ and $||\sigma||_4^2$. Thus, the previous inequality is equivalent to
		$$||\sigma||_1\geq ||M||_2^{3}||M||_4^{-2},$$
		finishing the proof.
	\end{proof}
	
	We highlight that in case $M$ is a $C_4$-free Boolean matrix, in which each row and column has approximately $d$ one-entries, then both the trace bound and the previous lemma implies the lower bound $\Omega(\sqrt{d})$. By a more involved argument, Balla, Hambardzumyan and Tomon \cite{BHT} sharpened this:  if $M$ is a $C_4$-free Boolean matrix with degeneracy $D$, then $\gamma_2(M)=\Theta(\sqrt{D})$.
	
	Now let us consider point-line incidences. If $M$ is the incidence matrix of $n$ points with respect to lines, then $M$ is $C_4$-free, making the previous inequalities very effective. In particular, Chazelle and Lvov \cite{ChL} used the trace bound to show that $\disc(M)$ can be as large as $\Omega(n^{1/6})$, by considering the extremal configuration for the Szemer\'edi-Trotter theorem \cite{SzT}. Up to a logarithmic loss, this also follows from our lower bounds on the $\gamma_2$-norm, as highlighted in \cite{BHT}. 
	
	\subsection{Dimension 3 via four-cycles}\label{sect:dim3}
	
	In this section, we demonstrate the use of our new trace bound (\Cref{lemma:trace_bound}) to prove lower bounds on the $\gamma_2$-norm of point-plane incidence matrices in $\mathbb{R}^3$. 
	
	\begin{theorem}\label{thm:3dim}
		There is a set of at most $n$ points and a set of at most $n$ hyperplanes in $\mathbb{R}^3$ such that the incidence matrix $M_{\text{hyp}}$ satisfies
		$$\gamma_2(M_{\text{hyp}})=\Omega\left(\frac{n^{1/4}}{\sqrt{\log n}}\right).$$
	\end{theorem}
	
	\noindent
	The rest of this section is devoted to the proof of this theorem. Let $m=\lfloor \frac{1}{2}n^{1/4}\rfloor-1$, let 
	$$Q=\{-m,\dots,m\}^4\setminus \{0\},$$ and let $M$ be the adjacency matrix of the orthogonality graph of $Q$. Formally, the rows and columns of $M$ are indexed by $Q$, and for $x,y\in Q$,
	$$M(x,y)=\begin{cases} 1&\text{if }\langle x,y\rangle=0,\\
		0&\text{otherwise}.\end{cases}$$
	More generally, we also define $Q_{d,T}=\{-T,\dots,T\}^d\setminus \{0\}$, so $Q=Q_{4,m}$.
	First, we make the simple observation that $M$ is the blow-up of an incidence matrix of points and hyperplanes in $\mathbb{R}^3$. In particular, this implies:
	
	\begin{claim}
		There is a set of at most $n$ points and a set of at most $n$ hyperplanes in $\mathbb{R}^3$ such that the incidence matrix $M_{\text{hyp}}$ satisfies
		$$\gamma_2(M_{\text{hyp}})=\gamma_2(M).$$
	\end{claim}
	
	\begin{proof}
		$M$ is the incidence matrix of a multiset of points and hyperplanes in the $3$-dimensional real projective space. Therefore, after applying a generic projection to the 3-dimensional affine space, there exists a set of at most $n$ points $P$ and a set of at most $n$ hyperplanes $\mathcal{H}$ in $\mathbb{R}^3$ such that if $M_{\text{hyp}}$ is the incidence matrix of $P$ and $\mathcal{H}$, then $M$ is a blow-up of $M_{\text{hyp}}$. But then $\gamma_2(M)=\gamma_2(M_{\text{hyp}})$.
	\end{proof}

	Therefore, our goal in this section is to prove lower bounds on $\gamma_2(M)$. First, we discuss some basic properties of $Q_{d,T}$, namely prove bounds on the number of orthogonal pairs. Say that an element of $\mathbb{Z}^d$ is \emph{primitive} if the greatest common divisor of its coordinates is 1.  
	
	\begin{claim}\label{claim:primitive}
		If $x\in Q_{d,T}$ is primitive and $||x||_{\infty}=q$, then there are $O_d(T^{d-1}/q)$ vectors $y\in Q_{d,T}$ such that $\langle x,y\rangle=0$.
	\end{claim}
	\begin{proof}
		Without loss of generality, let $|x_d|=||x||_{\infty}=q$. If $x_1=\dots=x_{d-1}=0$, then $q=1$ and the claim is true. So assume that not all of $x_1,\dots,x_{d-1}$ are 0. Then considering $\langle x,y\rangle=0$ modulo $q$, there are $q^{d-2}$ choices for $z\in \mathbb{Z}_q^{d-1}$ such that $\langle z,(x_1,\dots,x_{d-1})\rangle\equiv 0 \pmod{q}$. This is true due to $\text{gcd}(x_1,\dots,x_{d-1},q)=1$. Thus, counting $\#\{y\in Q:\langle x,y\rangle=0\}$ by the value of $(y_1,\dots,y_{d-1})\pmod{q}$, we get $O_d(q^{d-2}\cdot (T^{d-1}/q^{d-1}))=O_d(T^{d-1}/q)$.
	\end{proof}
	
	\begin{lemma}\label{lemma:num_pairs_upper}
		For $d\geq 3$, there are $O_d(T^{2d-2})$ pairs $(x,y)\in Q^2_{d,T}$ such that $\langle x,y\rangle=0$.
	\end{lemma}
	
	\begin{proof}
		The number of primitive vectors $x\in Q_{d,T}$ with $||x||_{\infty}=q$ is $O_d(q^{d-1})$, and by the previous claim, each contribute $O_d(T/q\cdot T^{d-1}/q)=O(T^d/q^2)$ pairs $(x',y)$, where $x'\in Q_{d,T}$ is some integer multiple of $x$. Therefore, the total number of pairs $(x,y)\in Q^2_{d,T}$ such that $\langle x,y\rangle=0$ is
		$$\sum_{q=1}^T O\left( q^{d-1}\cdot \frac{T^{d}}{q^2}\right)=O(T^{2d-2}).$$
	\end{proof}
	
	\begin{lemma}\label{lemma:num_pairs_lower}
		There are $\Omega_d(T^{2d-2})$ pairs $(x,y)\in Q^2_{d,T}$ such that $\langle x,y\rangle=0$.
	\end{lemma}
	
	\begin{proof}
		Let $T_0=\lfloor T/2\rfloor$, and for $0\leq r\leq dT_0^2$, let
		$$A_r=\#\{u\in \{-T_0,\dots,T_0\}^d:||u||_2^2=r\}.$$
		Then $$\sum_{r=0}^{dT_0^2} A_r=(2T_0+1)^{d},$$ so by the inequality between the arithmetic and quadratic mean,
		$$\sum_{r=0}^{dT_0^2} A_r^2
		\ge \frac{(2T_0+1)^{2d}}{dT_0^2+1}
		=\Omega_d(T^{2d-2}).$$
		Thus there are $\Omega_d(T^{2d-2})$ ordered pairs $(u,v)$ with $||u||_2^2=||v||_2^2$ and $u\not\in\{v,-v\}$, as the number of pairs with $u=\pm v$ is $O(T^d)$. If  $(u,v)$ satisfies $||u||_2^2=||v||_2^2$, let
		$x=u+v$ and $y=u-v$.
		Then $$\langle x,y\rangle=\langle u+v,u-v\rangle
		=||u||_2^2-||v||_2^2=0.$$
		Moreover, every coordinate of $x$ and $y$ lies in $[-2T_0,2T_0]\subseteq[-T,T]$. Thus, every pair $(u,v)$ with $||u||_2^2=||v||_2^2$ gives a unique pair $(x,y)\in Q^2_{d,T}$ such that $\langle x,y\rangle=0$.
	\end{proof}
	
	In order to apply \Cref{lemma:trace_bound}, we need to upper bound $||M||_4^4=\tr((M^*M)^2)$. Here, $||M||_4^4$ is exactly the number of homomorphisms of $C_4$ in the corresponding graph, or equivalently, the number of 4-tuples of vectors $(x_1,x_2,y_1,y_2)\in Q^4$ such that $$\langle x_i,y_j\rangle=0\quad \text{for }i,j\in \{1,2\}.$$

	\begin{lemma}\label{lemma:4cyclecount}
		$$||M||_4^4=\tr((M^*M)^2)=O(m^8 \log m).$$
	\end{lemma}
	
	\begin{proof}
		Let $F$ be the set of 4-tuples $(x_1,x_2,y_1,y_2)\in Q^4$ such that $$\langle x_i,y_j\rangle=0\quad \text{for }i,j\in \{1,2\}.$$
		For $x\in Q$, write $N(x)=\{y\in Q:\langle x,y\rangle=0\}$. 
		
		First, we consider the set of 4-tuples in which $x_1,x_2$ or $y_1,y_2$ are linearly dependent. Let $F_1\subset F$ be the set of 4-tuples $(x_1,x_2,y_1,y_2)\in F$ for which $x_1$ and $x_2$ span a 1-dimensional subspace, or $y_1$ and $y_2$ span a 1-dimensional subspace. The number of primitive elements $x\in Q$ with $||x||_{\infty}=q$ is $O(q^3)$, each such element spans a 1-dimensional subspace containing $O(m/q)$ elements of $Q$. Moreover, by \Cref{claim:primitive}, we have $|N(x)|=O(m^3/q)$. Therefore,
		$$|F_1|\ll \sum_{q=1}^{m} q^3\cdot \left(\frac{m}{q}\right)^2\cdot \left(\frac{m^{3}}{q}\right)^2= \sum_{q=1}^m \frac{m^{8}}{q}\ll m^8 \log m.$$
		
		Now consider the number of 4-tuples $(x_1,x_2,y_1,y_2)\in F$ where $x_1$ and $x_2$ are linearly independent, and $y_1$ and $y_2$ are linearly independent. Let $U$ be a 2-dimensional rational subspace of $\mathbb{R}^4$, and let $p(U)$ be its primitive Pl\"ucker vector. Namely, let $u,v\in U$ be a basis of the lattice $U\cap \mathbb{Z}^4 $. Then the Pl\"ucker vector $p(U)=(p_{12},p_{13},p_{14},p_{23},p_{24},p_{34})$ is defined as
		$$p_{ij}=\det\begin{pmatrix} u_i & u_j \\ v_i & v_j \end{pmatrix}.$$
		Up to sign, the value of $p(U)$ does not depend on the choice of the basis $u,v$, and $p(U)$ uniquely determines $U$. We collect some basic properties of $p(U)$:
		\begin{enumerate}
			\item The Pl\"ucker relation gives
			$$p_{12}p_{34}-p_{13}p_{24}+p_{14}p_{23}=0.$$
			\item Let $h(U)$ denote the area of the empty parallelogram of $U\cap \mathbb{Z}^4$. Then 
			$$h(U)=||p(U)||=\Theta(\max_{i,j}|p_{ij}|).$$
			\item Let $U^{\perp}$ denote the subspace orthogonal to $U$. Then
			$$p(U^{\perp})=\pm(p_{34},-p_{24},p_{23},p_{14},-p_{13},p_{12}).$$
			In particular, $h(U^{\perp})=h(U)$.
		\end{enumerate}
		From these properties, we deduce the following counting statements.
		\begin{claim}
			The number of 2-dimensional subspaces $U$ such that $||p(U)||_{\infty}\leq T$ is   $O(T^4)$. 
		\end{claim}
		
		\begin{proof}
			By the Pl\"ucker relation, we have $p_{12}p_{34}-p_{13}p_{24}+p_{14}p_{23}=0$. Therefore, by taking $x=(p_{12},p_{13},p_{14})$ and $y=(p_{34},-p_{24},p_{23})$, we get $\langle x,y\rangle=0$ and $(x,y)\in Q_{3,T}^2$. There are $O(T^3)$ such pairs where either $x$ or $y$ is 0, and by \Cref{lemma:num_pairs_upper}, the number of such pairs where $x$ and $y$ are non-zero is $O(T^4)$.
		\end{proof}
		
		\begin{claim}
			If $U$ contains two independent elements of $Q$, then $|U\cap Q|\leq O(\frac{m^2}{h(U)})$
		\end{claim}
		
		\begin{proof}
			Let $u$ be a shortest non-zero vector of $U\cap \mathbb{Z}^4$, and write $\lambda=||u||\leq O(m)$. Then $u$ can be extended to a lattice basis $u,v$ of $U\cap \mathbb{Z}^4$. Let $\delta$ be the distance of $v$ from the line $\mathbb{R}u$. As $Q\cap U$ contains an element linearly independent from $u$, we must have $\delta=O(m)$. Moreover, $h(U)=\lambda \delta$. The intersection $U\cap Q$ is contained in the union of lines $kv+\mathbb{R}u$ for $k\in \mathbb{Z}$. On each such line, the points are separated by distance $\lambda$, so as $Q\subset [-m,m]^4$, each such line contains at most $O(1+m/\lambda)=O(m/\lambda)$ points of $U\cap Q$. On the other hand, these lines are separated by distance $\delta$, so at most $O(1+m/\delta)=O(m/\delta)$ of these line intersect $[-m,m]^4$. Thus, 
			$$|U\cap Q|\leq O\left(\frac{m}{\lambda}\cdot\frac{m}{\delta}\right)=O\left(\frac{m^2}{h(U)}\right).$$
		\end{proof}
		
		Let $T>0$, and let $N_T$ be the number of 4-tuples  $(x_1,x_2,y_1,y_2)\in F\setminus F_1$ such that $(x_1,x_2)$ spans a 2-dimensional subspace $U$ with $T/2\leq h(U)\leq T$. First, as $h(U)=\Theta(||p(U)||_{\infty})$, there are $O(T^4)$ choices for $U$. Then each such $U$ contains $O((m^2/h(U))^2)=O(m^4/T^2)$ pairs of elements $(x_1,x_2)\in Q^2$. The orthogonal complement $U^{\perp}$ satisfies $h(U^{\perp})=h(U)$, so $U^{\perp}$ also contains $O(m^4/T^2)$ pairs of vectors $(y_1,y_2)\in Q^2$. Thus, 
		$$N_{T}=O\left(T^4\cdot \frac{m^4}{T^2}\cdot \frac{m^4}{T^2}\right)=O(m^8).$$
		Furthermore, for each $U$ intersecting $Q$ in two independent vectors, we have $h(U)=O(m^2)$. Thus, we have
		$$|F\setminus F_1|\leq \sum_{j=1}^{O(\log_2 m)} N_{2^j}=O(m^8\log m).$$
		This finishes the proof.
	\end{proof}

	\begin{proof}[Proof of \Cref{thm:3dim}]
		By \Cref{lemma:num_pairs_lower}, we have $t=\tr(M^*M)=\Omega(m^{6})$, and by \Cref{lemma:4cyclecount}, $s=\tr((M^*M)^2)=O(m^8\log m)$. Therefore, by the trace bound (\Cref{lemma:trace_bound}), we get
		$$\gamma_2(M)\geq \frac{1}{m^4} t^{3/2}s^{-1/2}=\Omega\left(\frac{m}{\sqrt{\log m}}\right).$$
		As $n=|Q|=(2m+1)^4-1$, this implies that $\gamma_2(M)\geq \Omega(n^{1/4}/\sqrt{\log n}).$
	\end{proof}

	What goes wrong when we consider the natural generalization for higher dimensions, that is, the orthogonality matrix $M$ of $\{-m,\dots,m\}^{D}$ for dimension $D\geq 5$? A somewhat more involved calculation shows that $s=\tr((M^*M)^2)=\Theta_D(m^{4D-8})$, so the trace bound implies the lower bound $\gamma_2(M)=\Omega_D(m)=\Omega_D(n^{1/D})$. Therefore, our lower bound gets worse as the dimension grows. However, this is not surprising, and the inherent flaw is in the trace bound. Indeed, if $M$ is any $n\times n$ Boolean matrix, the trace bound $\Lambda=\frac{1}{n}t^{3/2}s^{-1/2}$ cannot beat the bound $n^{1/4}$. This is true for the following reason: let $\sigma_1$ be the largest singular value of $M$. Then $\sigma_1$ is an upper bound on the degeneracy of $M$, so we have the two inequalities $$t=\tr(M^*M)\leq \sigma_1 n\quad\text{and}\quad 2\sqrt{\sigma_1}\geq \gamma_2(M)\geq \Lambda.$$ On the other hand, $s=\tr((M^*M)^2)\geq \sigma_1^4$, so $\Lambda \leq (n/\sigma_1)^{1/2}$. In conclusion, $$\Lambda\leq \min\{2\sigma_1^{1/2},(n/\sigma_1)^{1/2}\}=O(n^{1/4}).$$

    Thus in order to break the $n^{1/4}$ barrier, we have to consider a different approach. We will still apply the trace-norm, but first we transform our incidence matrix to a non-Boolean form, where such a bound is not constrained by the barrier described above.

	\subsection{High dimension via Fourier transform}\label{sect:main_lower}
	
	In this section, we prove \Cref{thm:main_lower}. In particular, we prove the following theorem, which immediately implies \Cref{thm:main_lower} in light of \Cref{lemma:gamma_herdisc}.

	\begin{theorem}\label{thm:gamma_2_lower}
		There is a set of at most $n$ points and a set of at most $n$ hyperplanes in $\mathbb{R}^d$ such that the incidence matrix $M_{\text{hyp}}$ satisfies
		$$\gamma_2(M_{\text{hyp}})=\Omega_d\left(n^{\frac{1}{2}-\frac{1}{d+1}}\right).$$
	\end{theorem}
	
	\begin{proof}
		Let $r=d-1$, $m=\lfloor n^{1/(r+2)}/(4r)\rfloor$, and let $p$ be a prime such that $2rm^2<p<4rm^2$. Let $L=\{0,\dots,m-1\}$ and $Q=L^r\times \mathbb{F}_p$, then $|Q|=m^rp\leq n$. Define the matrix $M$, whose rows and columns are indexed by $Q$, and for $(a,b),(x,y)\in Q$, we have
		$$M((a,b),(x,y))=\begin{cases}1 &\text{if } y\equiv \langle a,x\rangle+b \pmod{p},\\
			0 &\text{otherwise.}\end{cases}$$
		Our goal is to show that $\gamma_2(M)=\Omega_d(n^{1/2-1/(d+1)})$. While $M$ is not an incidence matrix of points and hyperplanes, we can write it as $M=M_1+M_2$ such that $M_1,M_2$ are both incidence matrices of points and hyperplanes. Indeed, if $y\equiv \langle a,x\rangle+b\mod{p}$ and $b,y\in \{0,\dots,p-1\}$, then the condition $p>2rm^2$ ensures that either
		$$y=\langle a,x\rangle+b\quad\text{or}\quad y=\langle a,x\rangle+b-p.$$
		Hence, defining $M_1,M_2\in \mathbb{R}^{Q\times Q}$ such that 
		$$M_1((a,b),(x,y))=\mathds{1}_{y=\langle a,x\rangle+b}\quad\text{and}\quad M_2((a,b),(x,y))=\mathds{1}_{y=\langle a,x\rangle+ (b-p)},$$
		we have $M=M_1+M_2$, and both $M_1$ and $M_2$ are incidence matrices of at most $n$ points and at most $n$ hyperplanes in $\mathbb{R}^d$. By the subadditivity of the $\gamma_2$-norm, we have $\gamma_2(M_1)\geq \gamma_2(M)/2$ or $\gamma_2(M_2)\geq \gamma_2(M)/2$.
		
		\medskip
		
		The rest of the proof is dedicated to showing that $\gamma_2(M)=\Omega_d(n^{1/2-1/(d+1)})$. The main idea is to analyze $M$ through the discrete Fourier transform. In a matrix language, this means multiplying $M$ with certain unitary matrices. By a suitable choice of such matrices, we transform $M$ into a direct sum of $p$ smaller matrices, whose trace-norms we bound using the trace bound (\Cref{lemma:trace_bound}).
        
        Let $T$ be the matrix, whose rows and columns are indexed by $\mathbb{F}_p$, and for $b,y\in \mathbb{F}_p$, $$T(b,y)=p^{-1/2}e_p(-by),$$
		where $e_p(x)=e^{2\pi i x/p}$. Then $T$ is the matrix corresponding to the discrete Fourier transform over $\mathbb{F}_p$, and it is a unitary matrix. Let $I=I_{L^r}$ be the identity matrix with columns and rows indexed by $L^r$, and write $U=I\otimes T$. Then $U\in \mathbb{C}^{Q\times Q}$ is also a unitary matrix, and for $(a,b),(x,y)\in Q$, we have
		$$U((a,b),(x,y))=p^{-1/2}e_p(-by)\mathds{1}_{a=x}.$$ Consider the matrix $\hat{M}=UMU^*$. As multiplying by unitary matrices does not change the singular values, we have $||\hat{M}||_{\tr}=||M||_{\tr}$. On the other hand, for $(a,t),(x,s)\in Q$, we have
		\begin{align*}
			\hat{M}((a,t),(x,s))&=\sum_{(x',y)\in Q}\sum_{(a',b)\in Q} U((a,t),(x',y))\cdot M((x',y),(a',b))\cdot U^*((a',b),(x,s))\\
            &=\frac{1}{p}\sum_{(x',y)\in Q}\sum_{(a',b)\in Q} e_p(-ty)\mathds{1}_{a=x'}\cdot \mathds{1}_{b=\langle a',x'\rangle+y}\cdot e_p(bs)\mathds{1}_{a'=x}\\
            &=\frac{1}{p}\sum_{b,y\in \mathbb{F}_p}e_p(-ty) \mathds{1}_{b=\langle a,x\rangle+y} e_p(bs)\\
			&=\frac{1}{p}\sum_{y\in \mathbb{F}_p} e_p(-ty)e_p\left(s[\langle a,x\rangle+y]\right) \\
            &=\frac{e_p(s\langle a,x\rangle)}{p}\sum_{y\in \mathbb{F}_p}e_p((s-t)y)\\
			&=e_p(s\langle a,x\rangle)\cdot\mathds{1}_{s=t}.
		\end{align*}
		For $s\in \mathbb{F}_p$, let $N_t\in \mathbb{C}^{L^r\times L^r}$ be the matrix defined as $$N_s(a,x)=e_p(s\langle a,x\rangle).$$ The previous equality shows that $\hat{M}=\oplus_{s\in \mathbb{F}_p} N_s$ is the direct sum of the matrices $N_s$, $s\in \mathbb{F}_p$. This implies that
		$$||\hat{M}||_{\tr}=\sum_{s\in \mathbb{F}_p} ||N_s||_{\tr}.$$
		Now write $B_s\in \mathbb{C}^{L\times L}$ for the matrix defined as $$B_s(u,v)=e_p(suv)\quad\text{for } u,v\in L.$$
		Then $N_s=B_s^{\otimes r}$, so $||N_s||_{\tr}=||B_s||_{\tr}^r,$ and therefore
		$$||\hat{M}||_{\tr}=\sum_{s\in \mathbb{F}_p}||B_s||_{\tr}^r.$$
		We bound the trace-norms of the matrices $B_s$ by the trace bound (\Cref{lemma:trace_bound}). In order to do this, we need to bound the Schatten 2- and 4-norms of $B_s$. We have $||B_s||_2^2=\tr(B_s^*B_s)=|L|^2=m^2$, as every entry of $B_s$ has unit absolute value. On the other hand, we have
		$$||B_s||_4^4=\tr((B_s^*B_s)^2)=\sum_{u_1,u_2,v_1,v_2\in L} e_p(s(u_1-u_2)(v_1-v_2)).$$
		Instead of bounding the right-hand-side individually for every $s$, we bound the average. We can write this as
		\begin{align*}
			\sum_{s\in\mathbb{F}_p}||B_s||_4^4&=\sum_{s\in \mathbb{F}_p}\sum_{u_1,u_2,v_1,v_2\in L} e_p(s(u_1-u_2)(v_1-v_2))\\
			&=\sum_{u_1,u_2,v_1,v_2\in L}p\cdot\mathds{1}_{(u_1=u_2)\vee (v_1=v_2)}\leq  2p|L|^3=2pm^3.
		\end{align*}
		Therefore, the trace bound gives 
		$$||B_s||_{\tr}\geq ||B_s||_2^3 \cdot ||B_s||_4^{-2}=m^{3}\cdot ||B_s||_4^{-2}.$$
		In conclusion,
		\begin{align*}
			||\hat{M}||_{\tr}&=\sum_{s\in \mathbb{F}_p}||B_s||_{\tr}^r
            \geq \sum_{s\in \mathbb{F}_p}(m^3\cdot ||B_s||_4^{-2})^r\\
			&=pm^{3r}\cdot \frac{\sum_{s\in \mathbb{F}_p} ||B_s||_4^{-2r}}{p}
			\geq pm^{3r}\cdot \left(\frac{p}{\sum_{s\in \mathbb{F}_p} ||B_s||_4^4}\right)^{r/2}\\
			&\geq pm^{3r}\cdot (2m^3)^{-r/2}=2^{-r/2}pm^{3r/2}.
		\end{align*}
		The second inequality holds due to the power mean inequality 
        $$\frac{\sum_{i=1}^px_i}{p}\geq \left(\frac{\sum_{i=1}^p x_{i}^{-r/2}}{p}\right)^{-2/r}$$
        applied with $x_i=||B_i||_4^4>0$. Thus,
		$$\gamma_2(M)\geq \frac{1}{|Q|}||M||_{\tr}=\frac{1}{m^rp}||\hat{M}||_{\tr}\geq 2^{-r/2}m^{r/2}.$$
		Recalling that $m^rp=\Theta_r(n)$ and $p=\Theta_r(m^2)$, we have $m=\Theta(n^{1/(r+2)})=\Theta(n^{1/(d+1)})$, so
		$$\gamma_2(M)=\Omega_d\left(n^{\frac{1}{2}-\frac{1}{d+1}}\right).$$
	\end{proof}
	
	\section{Upper bound --- Hyperplanes}\label{sect:hyperplanes}
	
	In this section, we prove \Cref{thm:main_upper} in case $k=1$, which corresponds to incidences with respect to hyperplanes. Considering this special subcase  allows us to present our main ideas in a clean setting, without the additional technicalities arising from algebraic geometry.
	In particular, we prove the following upper bound on the $\gamma_2$-norm of point-hyperplane incidence matrices, which immediately implies the desired result in view of \Cref{lemma:gamma_herdisc}.
	
	\begin{theorem}\label{thm:upper_hyperplane}
		Let $d\geq 2$. There exists $\eps_d=\Omega(d^{-3})$ such that if $Q$ is a set of $n$ points in $\mathbb{R}^d$ and $M$ is the incidence matrix of $Q$ with respect to hyperplanes, then 
		$$\gamma_2(M)=O_d\left(n^{\frac{1}{2}-\frac{1}{2d}-\eps_d}\right).$$
	\end{theorem}
	
	A key component of our proof is space-partitioning, in particular we use the following simplicial partition result of Matou\v{s}ek \cite{Mat} as a black-box. In the case of algebraic sets of bounded-complexity, this will be replaced by the celebrated polynomial partitioning technique of Guth and Katz \cite{GK}.
	
	\begin{theorem}[Matou\v{s}ek \cite{Mat}]\label{thm:matousek}
		Let $Q$ be a set of $n$ points in $\mathbb{R}^d$,  and let $1 < r \leq n$ be a given parameter. Then $Q$ can be partitioned into $q \leq 2r$ subsets, $Q_1,...,Q_q$ such that for each $i\in [q]$,
		\begin{enumerate}
			\item  $n/(2r) \leq |Q_i|\leq n/r$,
			\item $Q_i$ is contained in the relative interior of a (possibly lower-dimensional) simplex $\Delta_i$,
			\item every hyperplane crosses (i.e., intersects but does not contain) at most $O(r^{1-1/d})$ of the simplices $\Delta_1,\dots,\Delta_q$.
		\end{enumerate} 
	\end{theorem}

	\begin{proof}[Proof of \Cref{thm:upper_hyperplane}]
		Let $f_D(m)$ denote the maximum of $\gamma_2(N)$ over all incidence matrices $N$ of $m$ points and a set of hyperplanes in $\mathbb{R}^D$. Assume that $d\geq 3$ and $f_{d-1}(n)=O_d(n^{\alpha_{d-1}})$ for some $\alpha_{d-1}\in (0,1)$. In particular, we can take $\alpha_2=1/6$, as the Szemer\'edi-Trotter theorem \cite{SzT} implies that an incidence matrix of $n$ points with lines is $O(n^{1/3})$-degenerate.
		
		Let $1\leq r\leq n$ be a parameter specified later, and apply \Cref{thm:matousek}. Then there exists a partition $Q_1,\dots,Q_q$  of $Q$ for some $q\leq 2r$ so that $n/(2r)\leq |Q_i|\leq n/r$, $Q_i$ is contained in the relative interior of a simplex $\Delta_i$, and every hyperplane crosses at most $O(r^{1-1/d})$ of these simplices. 
		
		Based on this partition, we divide the point-hyperplane incidences $(x\in H)$ into three categories. This then gives a decomposition of $M$ into a sum of three matrices, whose $\gamma_2$-norms we bound individually. Let $(x\in H)$ be an incidence and let $Q_j$ be the unique part containing $x$. We say that $(x\in H)$ is
		\begin{enumerate}
			\item  \emph{degenerate} if $Q_j\subset H$; 
			\item  \emph{spanning} if $Q_j\not\subset H$ and $H\cap Q_j$ affinely-spans $H$;
			\item \emph{hybrid} otherwise.
		\end{enumerate}
		We denote by $M_{\text{deg}},M_{\text{spn}},M_{\text{hyb}}$ the matrices recording degenerate, spanning, and hybrid incidences, respectively. More precisely, $M_{\text{deg}}(H,x)=1$ iff $(x\in H)$ is a degenerate incidence, and $M_{\text{spn}},M_{\text{hyb}}$ are defined analogously. Note that $$M=M_{\text{deg}}+M_{\text{spn}}+M_{\text{hyb}}.$$
		
		\begin{claim}
			$$\gamma_2(M_{\text{deg}})\leq f_d(2r).$$
		\end{claim}
		\begin{proof}
			Let $\mathcal{H}$ be a finite set of hyperplanes, whose incidence matrix with $Q$ is equal to $M$. For every $Q_i$ with an at most $(d-1)$-dimensional affine span, choose a generic point $y_i\in \text{affine-span}(Q_i)$. Then $y_i$ is contained in only those hyperplanes $H\in \mathcal{H}$ for which $Q_i\subset H$.  Let $N$ be the incidence matrix of the set of points $\{y_i\}_i$ and $\mathcal{H}$. Then $M_{\text{deg}}$  is a blow-up of $N$, so $\gamma_2(N)=\gamma_2(M_{\text{deg}})$. On the other hand, $N$ is the incidence matrix of at most $q$ points and a set of hyperplanes, so $\gamma_2(N)\leq f_d(q)$.
		\end{proof}
		
		\begin{claim}
			$$\gamma_2(M_{\text{spn}})\ll_d \left(\frac{n}{r}\right)^{\frac{d-1}{2}}.$$
		\end{claim}
		
		\begin{proof}
			We prove that each column of $M_{\text{spn}}$ contains $O((n/r)^{d-1})$ 1-entries, from which the claim follows immediately.
			
			Let $x\in Q_i\subset Q$. If $(x\in H)$ is a spanning incidence, then there exists an affinely independent set of $d$ points in $Q_i\cap H$. But then there exists also an affinely independent set $B_H$ of $d$ points in $Q_i\cap H$ that contains $x$. The main observation is that $B_H$ uniquely determines $H$. Thus, the number of 1-entries in the column indexed by $x$ is at most the number of choices for the set $B_H$, which is at most $\binom{|Q_i|-1}{d-1}=O((n/r)^{d-1})$.  
		\end{proof}
		
		\begin{claim}
			$$\gamma_2(M_{\text{hyb}})\ll_d r^{\frac{1}{2}-\frac{1}{2d}}f_{d-1}\left(\frac{n}{r}\right).$$
		\end{claim}
		
		\begin{proof}
			Let $M_i$ be the submatrix of $M_{\text{hyb}}$ spanned by the columns indexed by the elements of $Q_i$. If the incidence $(x\in H)$ is not spanning and $x\in Q_i$, then $Q_i\cap H$ is contained in some $(d-2)$-dimensional subspace $S_{i,H}$. Replacing each such $H$ with $S_{i,H}$, we observe that $M_i$ is an incidence matrix of $Q_i$ and $(d-2)$-dimensional flats in $\mathbb{R}^d$. Applying a generic projection to $\mathbb{R}^{d-1}$, we also get that $M_i$ is an  incidence matrix of $|Q_i|$ points and a set of hyperplanes in $\mathbb{R}^{d-1}$. Thus, 
			$$\gamma_2(M_i)\leq f_{d-1}(|Q_i|)\leq f_{d-1}\left(\frac{n}{r}\right).$$
			Now $M_{\text{hyb}}$ is the concatenation of the matrices $M_1,\dots,M_q$. As each hyperplane $H$ crosses at most $t$ of the simplices $\Delta_1,\dots,\Delta_q$ for some $t=O_d(r^{1-1/d})$, the row of $M_i$ indexed by $H$ is non-zero for at most $t$ indices $i$. This is exactly the situation where \Cref{lemma:concatenation} can be applied, giving 
			$$\gamma_2(M_{\text{hyb}})\leq \sqrt{t}f_{d-1}\left(\frac{n}{r}\right)=O_d\left(r^{\frac{1}{2}-\frac{1}{2d}}f_{d-1}\left(\frac{n}{r}\right)\right).$$
		\end{proof}
		\noindent
		Combining these three upper bounds, we arrive to the inequality
		$$\gamma_2(M)\leq \gamma_2(M_{\text{deg}})+\gamma_2(M_{\text{spn}})+\gamma_2(M_{\text{hyb}})\leq f_d(2r)+O_d\left(\left(\frac{n}{r}\right)^{\frac{d-1}{2}}+r^{\frac{1}{2}-\frac{1}{2d}}f_{d-1}\left(\frac{n}{r}\right)\right).$$
		Recall that $\alpha_{d-1}$ is chosen such that $f_{d-1}(m)\leq O_d(m^{\alpha_{d-1}})$ for every $m$. Let 
		$$\beta=1-\frac{d-1}{d^2-1-2d\alpha_{d-1}},$$ and set $r=n^{\beta}$. Then the previous inequality becomes
		$$\gamma_2(M)\ll_d f_d(2n^{\beta})+n^{(1-\beta)\frac{d-1}{2}}+n^{\frac{\beta}{2}-\frac{\beta}{2d}}n^{(1-\beta)\alpha_{d-1}}=f_d(2n^{\beta})+2n^{\frac{(d-1)^2}{2d^2-2-4d\alpha_{d-1}}}.$$
		Hence, we get the recursive inequality for $f_d(n)$ that
		$$f_d(n)\ll_d f_d(2n^{\beta})+n^{\frac{(d-1)^2}{2d^2-2-4d\alpha_{d-1}}}.$$
		Writing $f_d(n)=O_d(n^{\alpha_d})$, we have $f_d(2n^{\beta})=O_d(n^{\alpha_d \beta})=o(n^{\alpha_d})$, so the exponent
		$$\alpha_d=\frac{(d-1)^2}{2d^2-2-4d\alpha_{d-1}}$$
		suffices. Having $\alpha_2=1/6$, we get $\alpha_3=2/7$, $\alpha_4=63/178$. In general, straightforward (but tedious) calculations show that 
		$$\alpha_d=\frac{1}{2}-\frac{1}{2d}-\Theta\left(\frac{1}{d^3}\right).$$
	\end{proof}

	\section{Upper bound --- Algebraic sets}\label{sect:upper}

    In this section, we present the proof of \Cref{thm:main_upper}. But first, we discuss the  algebraic-geometric terminology used in our paper, and present some basic results.  As a general reference, we refer the reader to the book \cite{Cox}.

    \subsection{Algebraic geometry preliminaries}\label{sect:alg_geo}
	
	We consider algebraic varieties over $\mathbb C$. A variety in $\mathbb R^d$ refers to the real points of a complex algebraic variety in $\mathbb C^d$
	defined over $\mathbb R$, and its dimension and degree refer to the corresponding complex variety. For a family $\mathcal F\subseteq
	\mathbb C[X_1,\ldots,X_d]$, let
	\[
	Z_{\mathbb{C}}(\mathcal F)
	=
	\{x\in\mathbb C^d:f(x)=0\text{ for every }f\in\mathcal F\},
	\]
	and write $Z_{\mathbb{R}}(\mathcal F)=Z_{\mathbb{C}}(\mathcal F)\cap \mathbb{R}^d.$ An affine algebraic set is a subset of $\mathbb C^d$ of this form. An
	algebraic set $V$ is called \emph{irreducible} if it cannot be written
	as the union of two proper algebraic subsets. Every algebraic set has
	a unique decomposition
	\[
	V=V_1\cup\cdots\cup V_s
	\]
	into irreducible components.
	
	The \emph{dimension} of an irreducible algebraic variety $V$ is the
	largest integer $r$ for which there exists a strictly increasing chain
	\[
	V_0\subsetneq V_1\subsetneq\cdots\subsetneq V_r=V
	\]
	of irreducible algebraic subsets. Equivalently, it is the Krull
	dimension of the coordinate ring of $V$. For a reducible algebraic
	set, we define
	\[
	\dim V=\max_i\dim V_i.
	\]
	
	To define degree, identify $\mathbb C^d$ with the affine chart
	$\{X_0\ne0\}$ in projective space $\mathbb P^d(\mathbb C)$, and let
	$\overline V$ be the projective closure of $V$. If $V$ is irreducible
	of dimension $r$, its \emph{degree}, denoted by $\deg V$, is the number
	of points in  $\overline V\cap L,$ counted with algebraic multiplicity, where $L\subseteq\mathbb P^d$ is a generic projective linear subspace of codimension~$r$.
	
	If $V=V_1\cup\cdots\cup V_s$ is reducible, we use the (cumulative) degree
	\[
	\deg V:=\sum_{i=1}^s\deg V_i.
	\]
	In particular, the number of irreducible components of $V$ is at most
	$\deg V$.
	
	A \emph{hypersurface} is an algebraic set of the form
	\[
	Z_{\mathbb{C}}(P)=\{x\in\mathbb C^d:P(x)=0\}
	\]
	for some nonzero polynomial $P$. If $P$ is irreducible, then
	\[
	\dim Z_{\mathbb{C}}(P)=d-1
	\qquad\text{and}\qquad
	\deg Z_{\mathbb{C}}(P)=\deg P.
	\]
	Conversely, every irreducible codimension-one subvariety of affine
	space is a hypersurface. Indeed, its vanishing ideal is a height-one
	prime ideal in the unique factorization domain
	$\mathbb C[X_1,\ldots,X_d]$, and hence is generated by one
	irreducible polynomial.

	By a real algebraic set we mean a complex algebraic set $V$  defined over $\mathbb{R}$, together with its real locus 
	$V\cap \mathbb{R}$. Its dimension and degree refer to those of $V$.
	
	We use the following form of B\'ezout's theorem, see e.g. Heintz \cite{Heintz}.
	
	\begin{lemma}[Affine B\'ezout theorem]\label{lem:bezout}
		Let $V,W\subseteq\mathbb C^d$ be affine algebraic sets. Then
		$$\deg(V\cap W)\leq \deg(V)\deg(W).$$
		Moreover, if no irreducible component of $V$ is contained in $W$, then
		$$\dim(V\cap W)\leq \dim(V)-1,$$
		with the convention that the dimension of the empty set is $-1$.
	\end{lemma}
	
	We will mostly use the following special case. If $P,R\in
	\mathbb C[X_1,\ldots,X_d]$, the polynomial $P$ is irreducible, and
	$P\nmid R$, then
	\[
	\dim Z_{\mathbb{C}}(P,R)\le d-2
	\]
	and
	\[
	\deg Z_{\mathbb{C}}(P,R)\le (\deg P)(\deg R).
	\]
	All these statements may be applied to real points: if the
	polynomials have real coefficients, then
	\[
	Z_{\mathbb R}(P_1,\ldots,P_s)
	\subseteq
	Z_{\mathbb C}(P_1,\ldots,P_s),
	\]
	so the same algebraic dimension and degree bounds control the
	corresponding real incidence sets.
	
	The final technical lemma we need states that we can project a family of $D$-dimensional varieties into a $D+1$ dimensional space without altering its incidences with a finite set of points.
	
	\begin{lemma}[Projection lemma]\label{lemma:projection}
		Let $Q\subset \mathbb{R}^m$ be a finite set, and let $V_1,\dots,V_s\subset \mathbb{C}^m$ be irreducible affine varieties of dimension $D$, each defined over $\mathbb{R}$. Then there exists a real affine linear map $\pi:\mathbb{R}^m\rightarrow \mathbb{R}^{D+1}$ such that writing $\pi_{\mathbb{C}}$ for its complexification, and $W_i=\pi_{\mathbb{C}}(V_i)$, we have that
		\begin{enumerate}
			\item $\pi$ is injective on $Q$;
			\item $W_i$ is an irreducible variety in $\mathbb{C}^{D+1}$;
			\item $W_i$ is defined over $\mathbb{R}$;
			\item $W_i$ is a hypersurface, so $W_i=Z_{\mathbb{C}}(P_i)$ for some irreducible $P_i\in \mathbb{R}[x_1,\dots,x_{D+1}]$;
			\item $\deg(W_i)\leq \deg(V_i)$;
			\item for every $x\in Q$, $x\in V_i \iff \pi(x)\in W_i$.
		\end{enumerate}
	\end{lemma}
	
	\begin{proof}
		Let $V=V_1\cup\dots\cup V_s$. Then $V$ is a $D$-dimensional affine algebraic set defined over $\mathbb{R}$. By the linear form of Noether normalization, there exists a real
		affine-linear map
		$u:\mathbb R^m\longrightarrow\mathbb R^D$
		whose complexification restricts to a finite morphism on $V$ 
		 (see Theorem 13.3 in \cite{Eisenbud}), and hence on every $V_i$. In particular, for every
		$x\in Q$ and every $i$, the fibre
		\[
		F_{x,i}:=\{z\in V_i:u_{\mathbb C}(z)=u(x)\}
		\]
		is finite. Choose a real linear functional $\lambda$ such that $\lambda_{\mathbb C}(z)\neq\lambda(x)$ for every $x\in Q,\ i\in[s]$ with $x\notin V_i$, and every $z\in F_{x,i}$,
		and such that	$\lambda(x)\neq\lambda(y)$ whenever $x,y\in Q$, $x\neq y$.
		Such a functional exists, since these conditions exclude only
		finitely many proper real linear subspaces of the space of linear
		functionals. Set $\pi=(u,\lambda).$
		The second condition shows that $\pi$ is injective on $Q$.
		
		Since $u_{\mathbb C}|_{V_i}$ is finite, so is
		$\pi_{\mathbb C}|_{V_i}$. Indeed, if $\mathbb C[V_i]$ is finite over
		$\mathbb C[u_1,\ldots,u_D]$, then the same finite set of module
		generators shows that it is finite over
		$\mathbb C[u_1,\ldots,u_D,\lambda].$
		It follows that $W_i=\pi_{\mathbb C}(V_i)$ is closed, irreducible,
		and of dimension $D$. It is also defined over $\mathbb R$.	Thus $W_i\subseteq\mathbb C^{D+1}$ is an irreducible variety of codimension one. Its vanishing
		ideal is therefore a height-one prime ideal in the unique
		factorization domain
		$\mathbb C[x_1,\ldots,x_{D+1}],$
		and hence is generated by one irreducible polynomial. Since $W_i$ is
		defined over $\mathbb R$, this polynomial may be chosen to have real
		coefficients.
		
		Moreover, linear projection does not increase degree. Indeed, if
		$L\subseteq\mathbb C^{D+1}$ is a generic affine linear subspace of
		codimension $D$, then every point of $W_i\cap L$ has a preimage in
		$V_i\cap\pi_{\mathbb C}^{-1}(L).$
		The latter is a zero-dimensional linear section of $V_i$, and hence
		contains at most $\deg V_i$ points counted with multiplicity.
		Therefore $\deg W_i\leq\deg V_i.$
		Finally, if $x\in V_i$, then clearly $\pi(x)\in W_i$. Conversely, if
		$x\notin V_i$ and $\pi(x)\in W_i$, then there is some $z\in V_i$ with
		$\pi_{\mathbb C}(z)=\pi(x)$. Hence $z\in F_{x,i}$ and
		$\lambda_{\mathbb C}(z)=\lambda(x)$, contradicting the choice of
		$\lambda$. This proves the incidence equivalence.
	\end{proof}

    \subsection{Discrepancy with respect to algebraic surfaces}
    
	This section is devoted to the proof of the following theorem, which immediately implies \Cref{thm:main_upper} by using \Cref{lemma:gamma_herdisc}.
	
	\begin{theorem}\label{thm:variety}
		For every $D,k\geq 1$, there exists $\eps=\Omega(D^{-2}\binom{D+1+k}{k}^{-1})$ such that the following holds. Let $Q\subset \mathbb{R}^{d}$ of size $n$, and let $\mathcal{A}$ be a finite family of affine algebraic sets of dimension at most $D$ and degree at most $k$. Let $M$ be the incidence matrix of $Q$ and $\mathcal{A}$. Then
		$$\gamma_2(M)=O_{D,k}(n^{1/2-1/2(D+1)-\eps}).$$
	\end{theorem}
	
	The proof of \Cref{thm:variety} follows the same ideas as the proof of \Cref{thm:upper_hyperplane}, but we need to make a number of crucial modifications. First, we replace the simplicial partition (\Cref{thm:matousek}) with the so-called polynomial partition technique. Second, we discuss a notion of spanning sets for hypersurfaces.
	
	\bigskip
	
	\noindent
	\textbf{Polynomial partitioning.} The breakthrough paper of Guth and Katz \cite{GK} on the \emph{Distinct distances problem} introduced the polynomial partitioning technique. They proved that any set of $n$ points in $\mathbb{R}^d$ can be cut by the zero set of a degree $T$ polynomial into $O_d(T^d)$ parts such that each open cell contains at most $O_d(n/T^d)$ points.
	
	\begin{lemma}[Polynomial partitioning]\label{lemma:poly_partition}
		Let $Q\subset \mathbb{R}^d$ be a set of $n$ points, and let $1\leq T$ be a parameter. Then there exists a non-zero real polynomial $f\in \mathbb{R}[x_1,\dots,x_d]$ of degree at most $T$ such that each connected component of $\mathbb{R}^d\setminus Z_{\mathbb{R}}(f)$ contains at most $O_d(n/T^d)$ elements of $Q$.
	\end{lemma}
	
	An important feature of the partition induced by the zero set of $f$ is that no constant degree hypersurface can intersect too many cells. The following lemma can be found in Guth \cite{Guth}, and is a simple corollary of Theorem A.2 in Solymosi and Tao \cite{ST}.
	
	\begin{lemma}
		Suppose $V$ is a $D$-dimensional variety in $\mathbb{R}^d$ defined by $m$
		polynomial equations $P_j (x) = 0$ each of degree at most $k$. If $f$ is a polynomial of degree at most $T$, then $V$ intersects at most $O_{k,m,d}(T^D)$ connected components of $\mathbb{R}^d\setminus Z_{\mathbb{R}}(f)$.
	\end{lemma}
	
	We only use the following simple corollary of this lemma.
	
	\begin{corollary}[Cell-crossing bound]\label{lemma:cell_crossing}
		Let $V$ be a hypersurface of degree at most $k$ in $\mathbb{R}^d$. If $f\in \mathbb{R}[x_1,\dots,x_d]$ is a polynomial of degree at most $T$, then $V$ intersects at most $O_{d,k}(T^{d-1})$ connected components of $\mathbb{R}^d\setminus Z_{\mathbb{R}}(f)$.  
	\end{corollary}
	
	A new difficulty that arises with the polynomial partitioning technique is to deal with the points of $Q$ that end up on the zero set $Z_{\mathbb{R}}(f)$. Several papers developed different methods to handle this exceptional set \cite{FPSSZ,MP,TY}. We present a self-contained argument inspired by these papers. We partition our set in multiple rounds, in each round taking $T$ to be a large constant, and we keep track of the points that end up on the zero set in any round. This defines a \emph{partition tree}, whose leaf nodes correspond to connected regions that are disjoint from the union of the zero sets, while the inner nodes correspond to regions on the zero sets. 
	
	\bigskip
	
	\noindent
	\textbf{Spanning set.} An important part of the proof of \Cref{thm:upper_hyperplane} is to differentiate whether the intersection of a hyperplane $H$ with a set $Q_i\subset \Delta_i$ spans $H$ or not. In order to generalize this idea, we define a notion of \emph{spanning} for bounded degree algebraic surfaces. Let
	$$\mathcal{P}_{k}=\{ P\in \mathbb{R}[x_1,\dots,x_d] :\deg(P)\leq k\}.$$
	Then $\mathcal{P}_k$ is a vector space over $\mathbb{R}$ of dimension $$L_{d,k}:=\binom{d+k}{k}.$$
	For a finite set $S\subset \mathbb{R}^d$, define 
	$$\mathcal{I}_k(S):=\{P\in \mathcal{P}_k: P(x)=0\ \forall x\in S\}.$$
    Then $\mathcal{I}_k(S)$ is a subspace of $\mathcal{P}_k$.	
	\begin{definition}
		Given a polynomial $P$ and a set $S$, we say that $S$ \emph{spans} $P$ if 
		$$\mathcal{I}_{\deg(P)}(S)=\mbox{span}_{\mathbb{R}}\{P\}.$$
	\end{definition}

    The main observation making this definition useful is that any minimal spanning set of a polynomial $P$ has size less than $L_{d,\deg{P}}$. The proof of this fact is simple linear algebra.
	
	\begin{lemma}\label{lemma:spanning}
		Let $S$ be a set that spans $P$, and let $x_0\in S$. Then there exists some $T\subset S$ of size $L_{d,\deg(P)}-1$ that also spans $P$ and $x_0\in T$.
	\end{lemma}
	
	\begin{proof}
		Let $k=\deg(P)$. For a set $S_0\subset \mathbb{R}^d$, define the linear map $E_{S_0}:\mathcal{P}_k\rightarrow \mathbb{R}^{S_0}$ such that for $R\in \mathcal{P}_k$, $$E_{S_0}(R)=(R(x))_{x\in S_0}.$$
		Then $S$ spans $P$ if and only if the kernel of $E_S$ is exactly $\mathbb{R}P$. The latter implies that $$\rank(E_S)=L_{d,k}-1.$$
        But observe that $E_S$ corresponds to an $|S|\times L_{d,k}$ matrix, whose row labeled by $x_0$ is a non-zero vector. Therefore, one can find $L_{d,k}-1$ row vectors, including $x_0$, that span the row-space of $E_S$.  The corresponding subset $T\subset S$ suffices. 
	\end{proof}
	
	Now we are ready to prove the main theorem of this section.
\begin{proof}[Proof of \Cref{thm:variety}]
		Let $f_{D,k}(n)$ denote the maximum $\gamma_2$-norm of an incidence matrix of at most $n$ points and a family of dimension $D$ affine algebraic sets of degree at most $k$. Our goal is to prove that $f_{D,k}(n)=O_{D,k}(n^{\frac{1}{2}-\frac{1}{2(D+1)}-\eps_{D,k}})$ for some $\eps_{D,k}>0$ depending only on $D$ and $k$. We proceed by induction on $D$. In case $D=0$, we have $f_{0,k}(n)\leq \sqrt{k}$ as every 0-dimensional affine algebraic set of degree at most $k$ contains at most $k$ points. So we assume that $D\geq 1$ and that for every $K>0$, $$f_{D-1,K}(n)=O_{D,K}(n^{\frac{1}{2}-\frac{1}{2D}-\eps_{D-1,K}})=O_{D,K}(n^{\frac{1}{2}-\frac{1}{2D}}).$$
		
		We also introduce a more restricted version of $f_{D,k}(n)$. Let $g_{d,k}(n)$ denote the maximum $\gamma_2$-norm of the incidence of $n$ points in $\mathbb{R}^d$ and irreducible hypersurfaces of degree at most $k$. Setting $d=D+1$, we first bound $g_{d,k}(n)$, and then use it to bound $f_{D,k}(n)$.

		Let $Q\subset \mathbb{R}^d$ be a set of points and let $\mathcal{V}$ be a finite family of irreducible hypersurfaces whose incidence matrix $M$ attains the maximum $g_{d,k}(n)$. 
		
		Let $T$ be sufficiently large with respect to $d$ and $k$. We construct a partition tree of the point set $Q$ as follows. The root $o$ of the tree is associated with the set $Q_{o}=Q$. Assume that $v$ is a vertex of the tree with associated set $Q_v$. Apply the Polynomial partition lemma  (\Cref{lemma:poly_partition}) to $Q_v$ with parameter $T$. Then there exists a polynomial   $f_v\in \mathbb{R}[x_1,\dots,x_d]$ of degree at most $T$ such that each cell of $\mathbb{R}^d\setminus Z_{\mathbb{R}}(f_v)$ contains at most $|Q_v|/R$ elements of $Q_v$ for some $R=c_dT^d$, where $c_d\in (0,1)$ is a constant depending only on $d$. Also, every hypersurface $V\in \mathcal{V}$ crosses at most $B=C_{d,k}T^{d-1}$ cells of $\mathbb{R}^d\setminus Z_{\mathbb{R}}(f_v)$, where $C_{d,k}\geq 1$ is a constant depending only on $d$ and $k$. Define
		$$Z_v=Z_{\mathbb{R}}(f_v)\cap Q_v.$$
		For every cell $\omega$ of $\mathbb{R}^d\setminus Z_{\mathbb{R}}(f_v)$, we add child node $y$ to $v$ with associated set $Q_y=Q_v\cap \omega$. Then $|Q_y|\leq |Q_v|/R$. We repeat this process until the tree has height at most $\ell$ for some integer $\ell$ specified later.
		
		If a node $v$ is at depth $j$, then $|Q_v|\leq n/R^j$. Therefore, writing $r=R^{\ell}$, the set associated to every leaf node has size at most $|Q|/r$. Next, say that $V\in \mathcal{V}$ is \emph{active} at a vertex $v$ if $Q_v\cap V\neq \emptyset$. If $V$ is active at some vertex $v$, then it is active for at most $B$ children of $v$. Thus, $V$ is active for at most $B^j$ vertices at depth $j$.

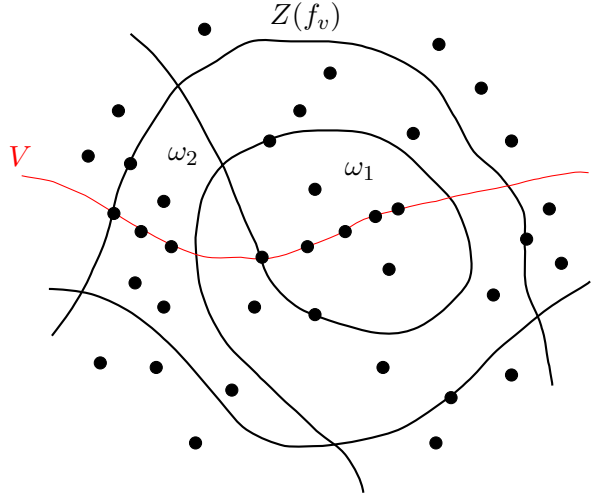
\begin{figure}
\begin{center}
\begin{tikzpicture}[scale=2.0]

  \draw[thick] (0.23,3.61) .. controls (0.26,3.58) and (0.36,3.51) .. (0.41,3.44) .. controls (0.47,3.37) and (0.53,3.29) .. (0.58,3.21) .. controls (0.63,3.13) and (0.68,3.04) .. (0.72,2.96) .. controls (0.76,2.88) and (0.81,2.8) .. (0.85,2.72) .. controls (0.89,2.63) and (0.91,2.55) .. (0.95,2.47) .. controls (0.98,2.38) and (1.02,2.3) .. (1.06,2.22) .. controls (1.09,2.13) and (1.11,2.04) .. (1.16,1.97) .. controls (1.21,1.9) and (1.27,1.84) .. (1.35,1.79) .. controls (1.42,1.75) and (1.52,1.72) .. (1.61,1.69) .. controls (1.69,1.67) and (1.77,1.64) .. (1.85,1.63) .. controls (1.93,1.62) and (2.01,1.62) .. (2.09,1.65) .. controls (2.16,1.67) and (2.24,1.73) .. (2.3,1.78) .. controls (2.36,1.84) and (2.44,1.91) .. (2.47,1.98) .. controls (2.5,2.06) and (2.48,2.15) .. (2.47,2.23) .. controls (2.45,2.31) and (2.42,2.39) .. (2.38,2.47) .. controls (2.34,2.54) and (2.27,2.6) .. (2.21,2.66) .. controls (2.15,2.72) and (2.09,2.79) .. (2.02,2.84) .. controls (1.95,2.89) and (1.87,2.93) .. (1.79,2.95) .. controls (1.71,2.97) and (1.62,2.97) .. (1.53,2.97) .. controls (1.45,2.97) and (1.36,2.97) .. (1.28,2.95) .. controls (1.19,2.93) and (1.12,2.9) .. (1.04,2.86) .. controls (0.97,2.83) and (0.88,2.79) .. (0.83,2.73) .. controls (0.77,2.67) and (0.73,2.59) .. (0.7,2.51) .. controls (0.68,2.43) and (0.67,2.34) .. (0.66,2.26) .. controls (0.66,2.17) and (0.67,2.08) .. (0.68,1.99) .. controls (0.69,1.91) and (0.72,1.84) .. (0.75,1.76) .. controls (0.79,1.68) and (0.83,1.6) .. (0.89,1.53) .. controls (0.95,1.46) and (1.03,1.4) .. (1.09,1.33) .. controls (1.16,1.27) and (1.21,1.21) .. (1.28,1.15) .. controls (1.34,1.1) and (1.4,1.04) .. (1.46,0.98) .. controls (1.53,0.92) and (1.61,0.88) .. (1.66,0.81) .. controls (1.71,0.74) and (1.76,0.65) .. (1.77,0.57) ;

  \draw[thick] (-0.29,1.61) .. controls (-0.26,1.65) and (-0.18,1.75) .. (-0.14,1.82) .. controls (-0.09,1.9) and (-0.05,1.97) .. (-0.02,2.05) .. controls (0.02,2.13) and (0.06,2.22) .. (0.08,2.3) .. controls (0.11,2.39) and (0.12,2.49) .. (0.15,2.58) .. controls (0.18,2.66) and (0.22,2.75) .. (0.26,2.83) .. controls (0.3,2.91) and (0.35,2.98) .. (0.4,3.06) .. controls (0.45,3.13) and (0.49,3.22) .. (0.56,3.28) .. controls (0.62,3.34) and (0.72,3.38) .. (0.8,3.43) .. controls (0.87,3.48) and (0.95,3.54) .. (1.03,3.56) .. controls (1.11,3.58) and (1.2,3.56) .. (1.29,3.56) .. controls (1.37,3.56) and (1.46,3.56) .. (1.54,3.56) .. controls (1.62,3.55) and (1.71,3.55) .. (1.79,3.52) .. controls (1.88,3.5) and (1.96,3.46) .. (2.04,3.42) .. controls (2.11,3.38) and (2.19,3.33) .. (2.24,3.27) .. controls (2.3,3.21) and (2.33,3.12) .. (2.38,3.05) .. controls (2.42,2.98) and (2.45,2.9) .. (2.5,2.83) .. controls (2.56,2.76) and (2.64,2.72) .. (2.69,2.65) .. controls (2.74,2.59) and (2.8,2.51) .. (2.83,2.43) .. controls (2.85,2.35) and (2.83,2.25) .. (2.84,2.17) .. controls (2.84,2.08) and (2.83,2) .. (2.85,1.92) .. controls (2.87,1.83) and (2.93,1.75) .. (2.95,1.66) .. controls (2.98,1.58) and (2.98,1.47) .. (3,1.41) .. controls (3.01,1.34) and (3.02,1.3) .. (3.02,1.28);
  \draw[fill] (2.35,1.2) circle (0.04cm);

  \draw[thick] (-0.31,1.92) .. controls (-0.26,1.92) and (-0.1,1.91) .. (-0.01,1.88) .. controls (0.08,1.85) and (0.16,1.81) .. (0.24,1.76) .. controls (0.31,1.71) and (0.38,1.66) .. (0.45,1.59) .. controls (0.52,1.53) and (0.59,1.45) .. (0.65,1.38) .. controls (0.71,1.31) and (0.76,1.25) .. (0.82,1.18) .. controls (0.88,1.11) and (0.94,1.02) .. (1.01,0.98) .. controls (1.09,0.93) and (1.17,0.9) .. (1.25,0.88) .. controls (1.33,0.87) and (1.42,0.88) .. (1.5,0.88) .. controls (1.59,0.89) and (1.67,0.9) .. (1.76,0.92) .. controls (1.85,0.94) and (1.95,0.96) .. (2.03,1) .. controls (2.11,1.04) and (2.18,1.09) .. (2.25,1.14) .. controls (2.32,1.19) and (2.39,1.24) .. (2.45,1.3) .. controls (2.52,1.36) and (2.57,1.43) .. (2.63,1.49) .. controls (2.69,1.55) and (2.76,1.6) .. (2.83,1.66) .. controls (2.9,1.72) and (2.97,1.78) .. (3.05,1.83) .. controls (3.12,1.88) and (3.24,1.94) .. (3.27,1.97);

  \draw[red] (-0.49,2.67) .. controls (-0.45,2.66) and (-0.35,2.65) .. (-0.28,2.63) .. controls (-0.21,2.6) and (-0.14,2.57) .. (-0.07,2.54) .. controls (0,2.5) and (0.06,2.46) .. (0.13,2.42) .. controls (0.19,2.38) and (0.25,2.34) .. (0.32,2.31) .. controls (0.38,2.27) and (0.46,2.23) .. (0.53,2.2) .. controls (0.6,2.17) and (0.67,2.14) .. (0.74,2.13) .. controls (0.82,2.12) and (0.89,2.13) .. (0.96,2.13) .. controls (1.03,2.12) and (1.11,2.11) .. (1.18,2.13) .. controls (1.25,2.14) and (1.32,2.17) .. (1.39,2.2) .. controls (1.46,2.22) and (1.53,2.25) .. (1.6,2.28) .. controls (1.67,2.32) and (1.73,2.37) .. (1.8,2.39) .. controls (1.87,2.42) and (1.95,2.43) .. (2.02,2.45) .. controls (2.09,2.47) and (2.17,2.49) .. (2.25,2.5) .. controls (2.32,2.52) and (2.39,2.54) .. (2.46,2.55) .. controls (2.54,2.57) and (2.63,2.58) .. (2.71,2.6) .. controls (2.79,2.61) and (2.85,2.64) .. (2.93,2.65) .. controls (3,2.67) and (3.1,2.68) .. (3.15,2.69) .. controls (3.21,2.7) and (3.24,2.69) .. (3.26,2.69);
  \node at (-0.5,2.8) {\color{red} $V$};
  
  \draw[fill] (0.66,0.9) circle (0.04cm);
  \draw[fill] (1.94,2.05) circle (0.04cm);
  \draw[fill] (3.08,2.09) circle (0.04cm);
  \draw[fill] (1.35,3.1) circle (0.04cm);
  \draw[fill] (0.72,3.64) circle (0.04cm);
  \draw[fill] (1.45,2.58) circle (0.04cm);
  \draw[fill] (1.9,1.4) circle (0.04cm);
  \draw[fill] (1.05,1.8) circle (0.04cm);
  \draw[fill] (0.26,1.96) circle (0.04cm);
  \draw[fill] (0.03,1.43) circle (0.04cm);
  \draw[fill] (2.27,3.54) circle (0.04cm);
  \draw[fill] (0.45,2.5) circle (0.04cm);
  \draw[fill] (1.55,3.35) circle (0.04cm);
  \draw[fill] (3,2.45) circle (0.04cm);
  \draw[fill] (2.75,1.35) circle (0.04cm);
  \draw[fill] (0.9,1.25) circle (0.04cm);
  \draw[fill] (0.4,1.4) circle (0.04cm);
  \draw[fill] (2.55,3.25) circle (0.04cm);
  \draw[fill] (2.1,2.95) circle (0.04cm);
  \draw[fill] (2.75,2.9) circle (0.04cm);
  \draw[fill] (-0.05,2.8) circle (0.04cm);
  \draw[fill] (0.15,3.1) circle (0.04cm);
  \draw[fill] (2.25,0.9) circle (0.04cm);
  \draw[fill] (2.85,2.25) circle (0.04cm);
  \draw[fill] (2.63,1.88) circle (0.04cm);

  \draw[fill] (0.12,2.42) circle (0.04cm);
  \draw[fill] (1.1,2.13) circle (0.04cm);
  \draw[fill] (0.23,2.75) circle (0.04cm);

  \draw[fill] (1.45,1.75) circle (0.04cm);

  \draw[fill] (1.15,2.9) circle (0.04cm);

  \draw[fill] (1.4,2.2) circle (0.04cm);

  \draw[fill] (1.65,2.3) circle (0.04cm);

  \draw[fill] (2,2.45) circle (0.04cm);

  \draw[fill] (1.85,2.4) circle (0.04cm);

  \draw[fill] (0.3,2.3) circle (0.04cm);

  \draw[fill] (0.5,2.2) circle (0.04cm);
  \node at (0.58,2.8) {$\omega_2$};
\node at (1.75,2.7) {$\omega_1$};

  \draw[fill] (0.45,1.8) circle (0.04cm);
  \node at (1.39,3.72) {$Z(f_v)$};
\end{tikzpicture}
\end{center}
\caption{An illustration of the space-cutting in dimension 2. Assume $v$ is a vertex of the partition tree at depth $\ell-1$. The union of the black curves is the zero-set of the cutting polynomial $f_v$; $\omega_1,\omega_2$ are two cells, and the red curve is an algebraic surface $V$. The incidences of $V$ with the points of $Q_v\cap Z(f_v)$ are in category 2.  The incidences of $V$ with the points of $Q_v\cap \omega_1$ are in category 3, because the points $\omega_1\cap Q_v \cap V$ span $V$. On the other hand, the incidences of $V$ with the points in $Q_v\cap \omega_2$ are in category 4, because $\omega_2\cap Q_v\cap V $ does not span $V$.}
\label{fig:1}
\end{figure}

		We divide the incidences into four categories. Let $(x\in V)$ be an incidence, where $x\in Q$ and $V\in \mathcal{V}$, and let $P$ be an irreducible polynomial such that $V=Z_{\mathbb{R}}(P)$. Then the four categories are
		\begin{enumerate}
			\item $x\in Z_v$ for some inner node $v$ and $P\mid f_v$;
			\item $x\in Z_v$ for some inner node $v$ and $P\nmid f_v$;
			\item $x\in Q_v$ for some leaf node $v$ and $V\cap Q_v$ spans $V$;
			\item $x\in Q_v$ for some leaf node $v$ and $V\cap Q_v$ does not span $V$.
		\end{enumerate}
		See \Cref{fig:1} for an illustration. For $i\in [4]$, let $M_i\in \{0,1\}^{\mathcal{V}\times Q}$ be the matrix recording the incidences in category $i$, that is, $M_i(V,x)=1$ iff $(x\in V)$ is an incidence of category $i$. Then $M=M_1+M_2+M_3+M_4$, and we bound $\gamma_2(M_i)$ for all four matrices individually.
		
		\begin{claim}
			$$\gamma_2(M_1)\leq2 B^{\ell/2}\sqrt{T}.$$
		\end{claim}
		
		\begin{proof}
			Let $v$ be an inner node of the partition tree, and let $N_v$ denote the submatrix of $M_1$, whose columns are indexed by the elements of $Z_v$. If $P\mid f_v$, then $P$ is one of the irreducible components of $f_v$. As the degree of $f_v$ is at most $T$, $f_v$ has at most $T$ irreducible components. This shows that $N_v$ has at most $T$ non-identical non-zero rows, so $\gamma_2(N_v)\leq \sqrt{T}$.
			
			Now for $j=0,\dots,\ell-1$, let $N_{j}$ be the concatenation of the matrices $N_v$ for the internal nodes $v$ on depth $j$. As every $V\in \mathcal{V}$ is active for at most $B^j$ nodes at depth $j$, we can apply \Cref{lemma:concatenation} to $M_j=(N_v)_{v\text{ is at depth }j}$ with parameters $t=B^j$ and $\gamma=\sqrt{T}$ to get
			$$\gamma_2(N_j)\leq B^{j/2}\sqrt{T}.$$
			Finally, as $M_1=(N_0\dots N_{\ell-1})$, we have
			$$\gamma_2(M_1)\leq \sum_{j=0}^{\ell-1}\gamma_2(N_j)\leq 2 B^{\ell/2}\sqrt{T}.$$
		\end{proof}
		
		\begin{claim}
			$$\gamma_2(M_2)\leq \sum_{j=0}^{\ell-1}B^{j/2}f_{D-1,kT}\left(\frac{n}{R^j}\right).$$
		\end{claim}
		\begin{proof}
			Let $v$ be a node of the partition tree, and let $N_v$ denote the submatrix of $M_2$, whose columns are indexed by the elements of $Z_v$. If $x\in Z_v\cap Z_{\mathbb{R}}(P)$, then $x\in Z_{\mathbb{R}}(f_v,P)$. By B\'ezout's theorem, if $P\nmid f_v$, then $Z_{\mathbb{C}}(f_v,P)$ has dimension at most $d-2$ and degree at most $(\deg P)(\deg f_v)\leq kT$. Thus, $N_v$ is an incidence matrix of $Z_v$ with affine algebraic sets of dimension at most $d-2=D-1$ and  degree at most $kT$. Therefore,
			$$\gamma_2(N_v)\leq f_{D-1,Tk}(|Z_v|)\leq f_{D-1,Tk}(|Q_v|).$$
			
			For $j=0,\dots,\ell-1$, let $N_{j}$ be the concatenation of the matrices $N_v$, where $v$ is an internal node of the partition tree at depth $j$. As every $V\in \mathcal{V}$ is active for at most $B^j$ such nodes, we can apply \Cref{lemma:concatenation} to $N_j$ with parameter $t=B^j$ to get
			$$\gamma_2(N_j)\leq B^{j/2}f_{D-1,kT}\left(\frac{n}{R^j}\right).$$
			Here, we used that $|Q_v|\leq n/R^j$ for every node $v$ at depth $j$.
			Finally, as $M_2=(N_0\dots N_{\ell-1})$, we have
			$$\gamma_2(M_2)\leq \sum_{j=0}^{\ell-1}\gamma_2(N_j)\leq \sum_{j=0}^{\ell-1}B^{j/2}f_{D-1,kT}\left(\frac{n}{R^j}\right).$$
		\end{proof}
		
		\begin{claim}
			$$\gamma_2(M_3)\leq O\left(\left(\frac{n}{R^{\ell}}\right)^{L_{d,k}/2-1}\right).$$
		\end{claim}
		\begin{proof}
			Let $v$ be a leaf vertex of the partition tree and let $x_0\in Q_v$. Recall that $|Q_v|\leq n/r=n/R^{\ell}$. Let $\mathcal{V}_{x_0}$ be the set of hypersurfaces $V\in \mathcal{V}$ such that $x_0\in V$ and $V\cap Q_v$ spans $V$. Then by \Cref{lemma:spanning}, there is a set $T_V\subset V\cap Q_v$ of size $L_{d,\deg(V)}-1$ that contains $x_0$ and spans $V$. As $T_V$ uniquely determines $V$, we get that 
			$$|\mathcal{V}_{x_0}|\leq \sum_{i=1}^k \binom{|Q_v|-1}{L_{d,i}-2}=O\left(\left(\frac{n}{r}\right)^{L_{d,k}-2}\right).$$
			This proves that every column of $M_3$ has at most $\left(\left(\frac{n}{r}\right)^{L_{d,k}-2}\right)$ one entries, implying
			$$\gamma_2(M_3)\leq O\left(\left(\frac{n}{r}\right)^{L_{d,k}/2-1}\right).$$
		\end{proof}
		
		\begin{claim}
			$$\gamma_2(M_4)\leq B^{\ell/2}f_{D-1,k^2}\left(\frac{n}{r}\right).$$
		\end{claim}
		
		\begin{proof}
			Let $v$ be a leaf vertex of the partition tree and let $N_v$ be the submatrix of $M_4$, whose columns are indexed by the elements of $Q_v$. Let $V\in \mathcal{V}$ such that $V\cap Q_v$ does not span $V$, and let $P$ be an irreducible polynomial such that $V=Z_{\mathbb{R}}(P)$. Then there exists some $P'\in \mathcal{I}_{\deg(P)}(V\cap Q_v)\setminus \text{span}_{\mathbb{R}}\{P\}$. As $\deg(P')\leq \deg(P)$, we have that $P\nmid P'$. Hence, by B\'ezout's theorem, $Z_{\mathbb{C}}(P,P')$ is an affine algebraic set of dimension at most $d-2$ and degree at most $\deg(P)\deg(P')\leq k^2$. Therefore, $N_v$ is the incidence matrix of $Q_v$ with affine algebraic sets of dimension at most $d-2=D-1$ and degree at most $k^2$, which gives
			$$\gamma_2(N_v)\leq f_{D-1,k^2}(|Q_v|)\leq f_{D-1,k^2}\left(\frac{n}{r}\right).$$
			Finally, $M_4$ is the concatenation of the matrices $N_v$, where $v$ is a leaf node. Each $V\in \mathcal{V}$ is active for at most $B^{\ell}$ leaf nodes, so we can apply \Cref{lemma:concatenation} with parameter $t=B^{\ell}$ to get
			$$\gamma_2(M_4)\leq B^{\ell/2}f_{D-1,k^2}\left(\frac{n}{r}\right).$$
		\end{proof}
		
		Next, we use our induction hypothesis $f_{D-1,k^2}(u)\leq f_{D-1,kT}(u)= O_{D,k,T}(u^{1/2-1/2(d-1)})$, and choose our parameters $T$ and $\ell$, which then define $r,B,R$ as well. Let $\mu=(100L_{d,k})^{-1}$ and $\delta=\mu(100d)^{-2}$. Set
		$$\theta:=\frac{\log B}{\log R}=\frac{(d-1)\log T+\log C_{d,k}}{d\log T+\log c_{d}}=\frac{d-1}{d}+O_{d,k}\left(\frac{1}{\log T}\right),$$
		and choose $T>1$ minimal such that $\theta\leq \frac{d-1}{d}+\delta$. Then $T$ depends only on $d$ and $k$. Choose $\ell$ minimal such that $r=R^{\ell}>n^{1-\mu}$. Then $r=\Theta_{d,k}(n^{1-\mu})$. Rewriting our previous four claims, we have 
		\begin{align*}
			1.\quad \gamma_2(M_1)&\leq 2 B^{\ell/2}\sqrt{T}=2\sqrt{T} R^{\theta \ell/2}\ll_{d,k} n^{(\frac{1}{2}-\frac{\mu}{2})(1-\frac{1}{d}+\delta)}\leq n^{\frac{1}{2}-\frac{1}{2d}-\delta},\\
			2.\quad\gamma_2(M_2)&\leq \sum_{j=0}^{\ell-1}B^{j/2}f_{D-1,kT}\left(\frac{n}{R^j}\right)\ll_{d,k} \sum_{j=0}^{\ell-1} R^{\theta j/2} \left(\frac{n}{R^j}\right)^{\frac{1}{2}-\frac{1}{2(d-1)}}\\
			&\leq \sum_{j=0}^{\ell-1} n^{\frac{1}{2}-\frac{1}{2(d-1)}} R^{\frac{j}{2}(\delta+\frac{1}{d-1}-\frac{1}{d})}\ll n^{\frac{1}{2}-\frac{1}{2(d-1)}} R^{\frac{\ell-1}{2}(\delta+\frac{1}{d-1}-\frac{1}{d})}\\
			&\ll n^{\frac{1}{2}-\frac{1}{2(d-1)}} n^{(\frac{1}{2}-\frac{\mu}{2})(\delta+\frac{1}{d-1}-\frac{1}{d})}=n^{\frac{1}{2}-\frac{1}{2d}+\frac{\delta}{2}-\frac{\mu\delta}{2}-\frac{\mu}{2(d-1)d}}\leq n^{\frac{1}{2}-\frac{1}{2d}-\delta}\\
			3.\quad\gamma_2(M_3)&\ll \left(\frac{n}{R^{\ell}}\right)^{L_{d,k}/2-1}\leq n^{\mu (L_{d,k}/2-1)}\leq n^{1/5}\\
			4.\quad\gamma_2(M_4)&\leq B^{\ell/2}f_{D-1,k^2}\left(\frac{n}{r}\right)\leq n^{\frac{1}{2}-\frac{1}{2d}-\delta}.
		\end{align*}
		Here, the calculations for $\gamma_2(M_4)$ are identical to the calculations for $\gamma_2(M_2)$. Thus, we proved that
		$$g_{d,k}(n)=\gamma_2(M)\leq \sum_{i=1}^{4}\gamma_2(M_i)\ll  n^{\frac{1}{2}-\frac{1}{2d}-\delta}.$$
		It remains to bound $f_{D,k}(n)$ in terms of $g_{d,k}(n)$. Let $Q$ be a set of points in $\mathbb{R}^m$ and let $\mathcal{A}$ be a finite family of affine algebraic sets of dimension at most $D$ and degree at most $k$ such that the incidence matrix $M$ of $Q$ and  $\mathcal{A}$ satisfies $\gamma_2(M)=f_{D,k}(n)$. For $V\in \mathcal{A}$, let
		$$V=X_1\cup\dots\cup X_s$$
		be the decomposition of $V$ into irreducible components over $\mathbb{C}$, then $s\leq k$. Since $V$ is defined over $\mathbb{R}$, complex conjugation permutes the $X_i$'s. In case $X_i\neq \overline{X}_i$, we have
		$$(X_i\cup \overline{X}_i)\cap \mathbb{R}^m=(X_i\cap \overline{X}_i)\cap \mathbb{R}^m$$
		and by B\'ezout's theorem, $\dim(X_i\cap \overline{X}_i)\leq D-1$ and $\deg(X_i\cap \overline{X}_i)\leq k^2$. Let 
		$$V\cap \mathbb{R}^m=Y_1\cup\dots\cup Y_t$$ be the decomposition, where each $Y_i$ is either a component $X_j\cap \mathbb{R}^m$ if $X_j=\overline{X}_j$, or $X_j\cap \overline{X}_j\cap \mathbb{R}^m$ if $X_j\neq \overline{X}_j$. Moreover, in case $t<k$, let $Y_{t+1},\dots,Y_{k}$ be empty sets. For $J\subset [k]$, define the set 
		$$V_J=\bigcap_{i\in J}Y_i,$$
		and let $M_J$ be the incidence matrix of $Q$ and the family $\{V_J: V\in \mathcal{A}\}$. Then the inclusion-exclusion formula gives
		$$M=\sum_{J\subset [k], J\neq \emptyset} (-1)^{|J|+1} M_J.$$
		Here, if $|J|\geq 2$, then $V_J$ is an affine algebraic set of dimension at most $D-1$ and degree at most $k^{2k}$ by repeated applications of B\'ezout's theorem. Thus, 
        $$\gamma_2(M_J)\leq f_{D-1,k^{2k}}(n)\ll_{D,k} n^{\frac{1}{2}-\frac{1}{2D}}\leq  n^{\frac{1}{2}-\frac{1}{2(D+1)}-\delta}.$$ In case $J=\{j\}$ for some $j\in [k]$, we can decompose $M_J$ into $M_{J,1}$ and $M_{J,2}$ along the rows depending on whether $Y_j$ has dimension $D$ or at most $D-1$, respectively. Then $\gamma_2(M_{J,2})\leq f_{D-1,k^2}(n)=O_{D,k}(n^{1/2-1/2D})$ as well. Finally, by the Projection lemma (\Cref{lemma:projection}), there is an affine linear projection $\pi:\mathbb{R}^m\rightarrow\mathbb{R}^{D+1}$ that maps the $D$-dimensional irreducible components $X_i$ (where $X_i=\overline{X}_i$) to irreducible hypersurfaces in $\mathbb{R}^{D+1}=\mathbb{R}^d$ of degree at most $k$ such that for $x\in Q$, we have $x\in Y_i\iff \pi(x)\in \pi(Y_i)$. Therefore, $M_{J,1}$ is the incidence matrix of $\pi(Q)$ and irreducible hypersurfaces in $\mathbb{R}^{d}$ of degree at most $k$, showing $$\gamma_2(M_{J,1})\leq g_{d,k}(n)\leq O_{D,k}\left(n^{\frac{1}{2}-\frac{1}{2(D+1)}-\delta}\right).$$
		In conclusion,
		$$\gamma_2(M)\leq \sum_{J\subset [k], J\neq \emptyset} \gamma_2(M_J)= O_{D,k}\left( n^{\frac{1}{2}-\frac{1}{2(D+1)}-\delta}\right).$$
		By observing that  $\delta\gg \frac{1}{L_{D+1,k} D^{2}}\gg D^{-2}\binom{D+1+k}{k}^{-1}$, this finishes the proof.
	\end{proof}
	
	\section{Communication complexity}\label{sect:comm_compl}
	
	Point-line incidence matrices have received considerable attention in communication complexity due to their remarkable algebraic properties \cite{CHHNPS,GH_totally_real,GHRS}. In particular, such matrices achieve large separation between the $\gamma_2$-norm and its approximate version, which in turn yields large separation between the randomized communication cost and deterministic communication cost with access to equality oracle.
	
	Given a Boolean matrix $M$ whose rows and columns are indexed by $m$-bit words $\{0,1\}^m$, Alice and Bob want to communicate the value of $M(x,y)$ such that Alice has access to only $x$, while Bob has access only to $y$. In \emph{public-coin randomized communication}, Alice and Bob share an infinite string of random bits, and they wish to communicate as few bits as possible so that they can guess the value of $M(x,y)$ with probability of success at least $2/3$. We denote by $R(M)$ the minimum number of bits in an optimal strategy. In \emph{deterministic communication}, Alice and Bob want to know the value of $M(x,y)$ without error, and \emph{access to equality oracle} means that a query of the form $\{s=t\}$ for arbitrary pairs of strings $(s,t)$ costs one unit. We denote by $D^{\text{EQ}}(M)$ the minimum cost of this communication protocol.  The relevance of the quantities to our paper is that $D^{\text{EQ}}(M)$ is controlled by the $\gamma_2$-norm of $M$, while $R(M)$ is controlled by the approximate $\gamma_2$-norm. Here, the approximate $\gamma_2$-norm is defined as
	$$\tilde{\gamma}_2(M)=\min \{\gamma_2(M'):\forall (x,y), |M'(x,y)-M(x,y)|\leq 1/3\}.$$
	Formally, we have
	$$\log \tilde{\gamma}_2(M)\ll R(M)\ll \tilde{\gamma}_2(M)^2\quad\cite{LinShr}\quad\text{ and }\quad D^{\text{EQ}}(M)\geq \frac{1}{2}\log_2 \gamma_2(M)\quad \cite{HHPTZ}.$$
	Linial and Shraibman \cite{LinShr} asked how much the $\gamma_2$-norm may differ from the approximate $\gamma_2$-norm, and thus from the randomized communication complexity.
	
	First, it was proved by Chattopadhyay, Lovett, and Vinyals that $R(M)$ and $D^{\text{EQ}}(M)$ might differ a lot by showing an $O(\log m)$ versus $\Omega(m)$ separation (where we remind the reader that the inputs are $m$-bit words, so the size of the matrix is $n=2^m$). Later Cheung,  Hatami,  Hosseini,  Nikolov,  Pitassi, and  Shirley \cite{CHHNPS} showed that the same separation is achieved by point-line incidences defined over integer grids. More precisely, this construction achieves $\gamma_2(M)=\Omega(n^{1/6})$ and $\tilde{\gamma}_2(M)=O(\log n)$. On the other hand, G\"o\"os, Harms, Richter, and Sofronova \cite{GHRS} proved that these matrices cannot achieve constant cost randomized communication. To this end, Hambardzumyan, Hatami, and Hatami \cite{HHH} proved that Hamming-distance matrices achieve $R(M)=O(1)$ and $D^{\text{EQ}}(M)=\Omega(\log m)$, while G\"o\"os, Harms, and Riazanov \cite{GHR} provided matrices showing a separation of $O(1)$ versus $\Omega(\sqrt{m})$. In particular, in terms of the size $n$, their matrix achieves $\gamma_2(M)=e^{\Omega(\sqrt{\log n})})$. In a recent breakthrough, Goh and Hatami \cite{GH_totally_real} proved that point-line incidence matrices over grids of algebraic integers can achieve extreme separation. They provided such a matrix with $R(M)=O(1)$, $D^{\text{EQ}}(M)=\Omega(m)$, and $\gamma_2(M)\geq n^{1/6-\eps}$ for any constant $\eps>0$. The goal of this section is to push the ideas of \cite{GH_totally_real} to the limit, and show that point-hyperplane incidence matrices over grids of algebraic integers can achieve $R(M)=O(1)$ and $\gamma_2(M)\geq n^{1/2-\eps}$. We highlight that the maximum $\gamma_2$-norm of an $n\times n$ Boolean matrix is $O(n^{1/2})$.

	\begin{theorem}\label{thm:cc}
		Let $d$ be a positive integer and $c>0$, then the following holds for infinitely many positive integers $n$. There exists an incidence matrix $M$ of at most $n$ points and  $n$ hyperplanes in $\mathbb{R}^d$ such that 
		$$R(M)=O_{d,c}(1)\quad\text{ and }\quad\gamma_2(M)\geq n^{\frac{1}{2}-\frac{2}{d}-c}.$$
	\end{theorem}

    \begin{corollary}
    For every $\eps>0$, there exist infinitely many $n$ and $n\times n$ Boolean matrices $M$ such that
    $$R(M)=O_{\eps}(1)\quad\text{ and }\quad\gamma_2(M)\geq n^{\frac{1}{2}-\eps}.$$
    \end{corollary}
	
	We follow the proof of Goh and Hatami \cite{GH_totally_real}, which in turn is based on the recent refutation of the Sum-Product conjecture \cite{BSSZ} using totally real number fields. Our new input in the proof is a lower bound on the $\gamma_2$-norm of dense point-hyperplane incidence matrices. In \cite{GH_totally_real}, the authors only considered point-line incidences, in which case one can invoke the result of \cite{BHT} on the $\gamma_2$-norm of $C_4$-free matrices to immediately get an optimal bound. However, this approach no longer works for higher-dimensional point-hyperplane incidences, so we employ analytic and linear algebraic techniques instead. Despite \Cref{thm:cc} being very similar to \Cref{thm:gamma_2_lower}, our proof is based on a different strategy.
	
	\subsection{Number fields}
	
	In this section, we give a brief introduction to number fields and the relevant results needed for our proof. We refer the reader to the book of Marcus \cite{Marcus} as a general reference.
	
	A \emph{number field} $K$ is a field containing $\mathbb{Q}$ that is finite dimensional over $\mathbb{Q}$ as a vector space. The degree of $K$ is this dimension. An \emph{algebraic integer} in $K$ is a root of a monic polynomial with integer coefficients, and we denote by $\mathcal{O}_K$ the set of algebraic integers. A degree $D$ number field $K$ admits $D$ injective field homomorphisms $\sigma_1,\dots,\sigma_D$ into $\mathbb{C}$. The number field $K$ is \emph{totally real} if $\sigma_1,\dots,\sigma_D$ map $K$ into $\mathbb{R}$. The \emph{norm} of $x\in K$ is 
	$$N(x)=\prod_{i=1}^D \sigma_i(x).$$
	
	\begin{lemma}[\cite{Marcus}]\label{lemma:norm}
		For every non-zero $x\in \mathcal{O}_K$, $|N(x)|\geq 1$. 
	\end{lemma}
	
	For a positive real number $R$, define the ball
	$$B_K(R)=\{x\in \mathcal{O}_K: \forall i\in [D], |\sigma_i(x)|\leq R\}.$$
	
	\begin{lemma}[\cite{BSSZ}]\label{lemma:ball_size}
		Let $\Delta_K$ be the discriminant of $K$. For every $R>0$,
		$$\frac{R^D}{\sqrt{\Delta_K}}\leq |B_K(R)|\leq (2R+1)^D.$$
	\end{lemma}
	
	We will not need the exact definition of the discriminant, we use $\Delta_K$ as an abstract constant associated with the field $K$. We also need the following result of Martinet \cite{Mart}, which is the key in all recent applications of totally real number fields.
	
	\begin{lemma}[\cite{Mart}]
		There is a constant $C>0$ such that for infinitely many integers $D$ there is a totally real number field $K$ of degree $D$ such that $\Delta_K\leq C^D$.
	\end{lemma}
	
	The final algebraic ingredient we need is the existence of certain characters over quotients of $\mathcal{O}_K$. Let $G$ be a finite abelian group, then a  \emph{character} is a group homomorphism $\chi:G\rightarrow \mathbb{C}^{\times}$. Formally, for every $g,h\in G$:  $\chi(g+h)=\chi(g)\chi(h)$. In case $\mathcal{R}$ is a finite commutative ring, a character of $\mathcal{R}$ refers to a character of its additive group. Note that if $\chi$ is a character, then for any $r\in \mathcal{R}$, the function $\chi'$ defined as $\chi'(u)=\chi(ru)$ is also a character. We say that $\chi$ is \emph{non-trivial} if $\chi\not\equiv 1$, and a character is \emph{generating} on a ring $\mathcal{R}$ if for every non-zero $u\in \mathcal{R}$ there exists some $v\in \mathcal{R}$ such that $\chi(uv)\neq 1$. For example, if $\mathcal{R}=\mathbb{Z}/p^2\mathbb{Z}$ for some prime $p$, the function $\chi(u)=e^{2\pi i u/p}$ is a non-trivial character, but  it is not generating because $\chi(pv)=1$ for every $v$. On the other hand, the character $\chi(u)=e^{2\pi i u/p^2}$ is generating. We use the following well known property of characters: 
	$$\sum_{g\in G} \chi(g)=\begin{cases} |G| &\text{if }\chi\text{ is trivial},\\
		0 &\text{otherwise}.\end{cases}$$
	In particular, if $\chi:\mathcal{R}\rightarrow\mathbb{C}^{\times}$ is a generating character, then for every $r\in \mathcal{R}$,
	$$\sum_{u\in\mathcal{R}}\chi(ru)=\begin{cases} |\mathcal{R}| &\text{if }r=0,\\
		0 &\text{otherwise}.\end{cases}$$
	
	\begin{lemma}\label{lemma:character}
		Let $q\geq 2$ be an integer, then the ring $\mathcal{O}_K/q\mathcal{O}_K$ has a generating character.
	\end{lemma}
	
	\begin{proof}
		It is well known that $\mathcal{O}_K$ is a Dedekind domain, which implies that every non-zero proper ideal has a unique factorization into prime ideals. In particular, we can write
		$$q\mathcal{O}_K=\mathfrak{p}_1^{e_1}\dots\mathfrak{p}_t^{e_t}$$
		with unique distinct prime ideals $\mathfrak{p}_1,\dots,\mathfrak{p}_t<\mathcal{O}_K$. By the Chinese remainder theorem,
		$$\mathcal{O}_K/q\mathcal{O}_K\cong\prod_{i=1}^t (\mathcal{O}_K/\mathfrak{p}_i^{e_i}).$$
		Therefore, it is enough to find a generating character $\chi_i:\mathcal{O}_K/\mathfrak{p}_i^{e_i}\rightarrow \mathbb{C}^{\times}$. Indeed,  taking any isomorphism $\phi:\mathcal{O}_K/q\mathcal{O}_K\rightarrow\prod_{i=1}^t (\mathcal{O}_K/\mathfrak{p}_i^{e_i})$, the character $\chi(u):=\prod_{i=1}^t\chi_i(\phi(u)_i)$ is a generating character of $\mathcal{O}_K/ q\mathcal{O}_K$.
		
		Let $\mathfrak{p}$ be a prime ideal and consider $\mathcal{R}=\mathcal{O}_K/\mathfrak{p}^{a}$. The ideals of $\mathcal{R}$ form the chain $$\mathcal{R}\supset \mathfrak{p}/\mathfrak{p}^{a}\supset\dots\supset \mathfrak{p}^{a-1}/\mathfrak{p}^{a}\supset 0.$$ In particular, the unique smallest non-zero ideal of $\mathcal{R}$ is $S=\mathfrak{p}^{a-1}/\mathfrak{p}^{a}$. Let
		$\kappa=\mathcal O_K/\mathfrak p,$
		which is a finite field of some prime characteristic $p$. Multiplication
		induces a natural $\kappa$-vector-space structure on $S$,
		since $\mathfrak p\,\mathfrak p^{a-1}\subseteq\mathfrak p^a.$
		In particular, $S$ is a nonzero finite-dimensional vector space over
		$\mathbb F_p$. Choose a nonzero $\mathbb F_p$-linear functional
		$\lambda:S\longrightarrow\mathbb F_p,$
		and define
		\[
		\chi_0(x)
		=
		\exp\left(\frac{2\pi i\lambda(x)}{p}\right).
		\]
		Then $\chi_0$ is a non-trivial character of $S$. We remark that $\chi_0$ is not necessarily a generating character of $S$. Next, we extend $\chi_0$ to a character $\chi:\mathcal{R}\rightarrow\mathbb{C}^{\times}$ arbitrarily, which we can do by the following claim.
		
		\begin{claim}
			If $H$ is a subgroup of the finite abelian group $G$ and $\chi_H:H\rightarrow\mathbb{C}^{\times}$ is a character, then there exists a character $\chi_G:G\rightarrow \mathbb{C}^{\times}$ such that $\chi_G|_H\equiv \chi_H$.
		\end{claim}
		
		\begin{proof}
			If $H=G$, there is nothing to prove. Otherwise, let $g\in G\setminus H$, and let $n$ be the smallest positive integer such that $ng\in H$. Choose any $\xi\in \mathbb{C}^{\times}$ such that $\xi^n=\chi_H(ng)$.  Every element $x\in H'=H+\langle g\rangle$ can be written as $x=h+mg$ for some $m\in \mathbb{Z}$ and $h\in H$. Define
			$$\chi_{H'}(x)=\xi^m\chi_H(h),$$
			then it is easy to check that the definition $\chi_{H'}(x)$ does not depend on the representation of $x$, and  $\chi_{H'}$ is a character of $H'$ such that $\chi_{H'}|_{H}\equiv \chi_H$. Continue this process until a character for the whole group $G$ is defined.
		\end{proof}
		
		Therefore, there exists a character $\chi:\mathcal{R}\rightarrow\mathbb{C}^{\times}$ such that $\chi|_S\equiv\chi_0$. We claim that $\chi$ is generating. Indeed, let $r\in \mathcal{R}$ be an arbitrary non-zero element. Then the ideal $(r)$ generated by $r$ contains $S$, so for every $v\in S$, there exists $u\in \mathcal{R}$ such that $ru=v$. As $\chi_0$ is non-trivial, we can choose $v$ such that $\chi_0(v)\neq 1$, and so $\chi(ru)\neq 1$. Thus, $\chi$ is generating. 
	\end{proof}
	
	\subsection{Orthogonality matrix over algebraic integers}
	In this section, we present the matrix satisfying the requirements of \Cref{thm:cc}.
	Pick $R$ sufficiently large with respect to the constant $C$ and the constants $d,c$ from \Cref{thm:cc}. Let $D$ be an integer for which a totally real number field $K$ exists such that $\Delta_K\leq C^D$. Let $Q=B_K(R)^d$ and let $M$ be the matrix, whose rows and columns are indexed by $Q$, and for $a,x\in Q$,
	$$M(a,x)=\begin{cases}1 &\text{if }\langle a,x\rangle=0,\\
		0 &\text{otherwise}.\end{cases}$$
	Let $n=|Q|$, then by \Cref{lemma:ball_size}, we have
	$$\left(\frac{R}{\sqrt{C}}\right)^{dD}\leq n\leq (2R+1)^{dD}\leq (3R)^{dD}.$$
    Here, $M$ is not necessarily an incidence matrix of points and hyperplanes, but it is very close to one. 
    \begin{claim}
    There exists a set of at most $n$ points and at most $n$ linear hyperplanes in $\mathbb{R}^d$ whose incidence matrix $N$ satisfies $\gamma_2(N)\geq \gamma_2(M)-2$.
    \end{claim}

    \begin{proof}
     Let $Q_0=Q\setminus \{0\}$ and let $Q_1$ be the set of primitive elements of $Q_0$. Then $N=[Q_1\times Q_1]$ is an incidence  matrix of at most $n$ points and at most $n$ hyperplanes in $\mathbb{R}^d$. Moreover, $M[Q_0\times Q_0]$ is a blow-up of $N$, so $\gamma_2(N)=\gamma_2(M[Q_0\times Q_0])$. As $M$ differs from $M[Q_0\times Q_0]$ in a single row and column, we also have $\gamma_2(M[Q_0\times Q_0])\geq \gamma_2(M)-2$.
    \end{proof}
	Next, we prove that $M$ has many one-entries.
	
	\begin{lemma}\label{lemma:num_one_entries}
		If $R$ is sufficiently large with respect to $C,c,d$, then $M$ has at least $n^{2-\frac{2}{d}-\frac{c}{2}}$ one-entries.
	\end{lemma}
	
	\begin{proof}
		For $z\in \mathcal{O}_K$, define
		$$\rho(z)=\# \{b\in B_K(R/2)^d : \langle b,b\rangle=z\}.$$
		If $\rho(z)\neq 0$, then $z\in B_K(d R^2)$, as for $i\in [D]$, we  have
		$$|\sigma_i(\langle b,b\rangle)|=\sum_{j=1}^d \sigma_i(b_j)^2\leq dR^2.$$
		Moreover, 
		$$\sum_{z}\rho(z)=|B_K(R/2)|^d.$$
		Thus, by the Cauchy-Schwarz inequality
		$$\sum_{z}\rho(z)^2\geq \frac{|B_K(R/2)|^{2d}}{|B_K(dR^2)|}\geq \frac{(R/2)^{2dD}/C^{dD}}{(2dR^2+1)^D}\geq \frac{R^{2dD-2D}}{4^{4dD}C^{dD}}\geq \frac{n^{2-\frac{2}{d}}}{12^{4dD}C^{2dD}},$$
		where the second inequality holds by \Cref{lemma:ball_size}, and the last inequality holds by the bound $n\leq (3R)^{dD}$.
		Assuming $R$ is sufficiently large with respect to $C,c,d$, we have
		$$\sum_{z}\rho(z)^2\geq n^{2-\frac{2}{d}-\frac{c}{2}}.$$
		Observe that $\rho(z)^2$ is the number of pairs $(b,y)$ such that $\langle b,b\rangle=\langle y,y\rangle=z$. Each such pair defines a unique pair $(a,x)\in Q^2$ by setting $a=b+y$ and $x=b-y$ that satisfies $\langle a,x\rangle=0$.
	\end{proof}
	
	Next, we prove that the $\gamma_2$-norm of $M$ is large. This is our main new contribution to the idea of Goh and Hatami \cite{GH_totally_real}.
	
	\begin{lemma}
		If $R$ is sufficiently large with respect to $C,c,d$, then $\gamma_2(M)\geq n^{\frac{1}{2}-\frac{2}{d}-c}.$
	\end{lemma}
	
	\begin{proof}
		We have $\gamma_2(M)\geq \frac{1}{n}||M||_{\tr}$ and 
		$$||M||_{\tr}=\max_{||A||_{\text{op}}\leq 1}\langle A,M\rangle.$$
		Therefore, in order to prove a lower bound on $\gamma_2(M)$, it is enough to find a test matrix $A$ with bounded operator norm such that $\langle M,A\rangle$ is large.
		
		Let $q$ be an integer such that $2R<q\leq 3R$ and consider
		$$\mathcal{L}=\mathcal{O}_K/q\mathcal{O}_K.$$
		It is standard that $\mathcal O_K$ is a free abelian group of rank
		$D$. Consequently, $|\mathcal{L}|=q^D$. Moreover, the map $B_K(R)\rightarrow \mathcal{L}$ is injective. Indeed, otherwise, there exists $a,b\in B_K(R)$ such that $a-b=qx$ for some non-zero $x\in \mathcal{O}_K$. But then $2R\geq |\sigma_i(a-b)|=q|\sigma_i(x)|$ for every $i\in [D]$. But we recall that \Cref{lemma:norm} ensures that at least for one $i\in [D]$, we have $|\sigma_i(x)|\geq 1$, so we get $2R\geq q$, contradiction.
		
		By \Cref{lemma:character}, there exists a generating character $\chi:\mathcal{L}\rightarrow \mathbb{C}^{\times}$. Define the matrix $X$, whose rows and columns are indexed by $\mathcal{L}$, and $X(u,v)=\chi(uv)$.
		Observe that
		\begin{align*}
			(X^*X)(u,v)&=\sum_{w\in \mathcal{L}} X^*(u,w)X(w,v)=\sum_{w\in \mathcal{L}}\overline{\chi(uw)}\chi(wv)\\
			&=\sum_{w\in \mathcal{L}}\chi(w(v-u))=\begin{cases}|\mathcal{L}|&\text{if }u=v\\ 0&\text{otherwise.}\end{cases}
		\end{align*}
		Therefore, $X^*X=|\mathcal{L}| I$, which implies that $||X||_{\text{op}}=\sqrt{|\mathcal{L}|}$. Now let $Y$ be the matrix, whose rows and columns are indexed by the Cartesian power $\mathcal{L}^d$, and $Y(x,y)=\chi(\langle x,y\rangle)$. Then $Y=X^{\otimes d}$, so $$||Y||_{\text{op}}=||X||_{\text{op}}^{d}=|\mathcal{L}|^{\frac{d}{2}}=q^{\frac{dD}{2}}.$$
		Identify $Q$ injectively  with a subset of $\mathcal{L}^d$, and set
		$$A:=q^{-\frac{dD}{2}} Y\left[ Q\times Q\right].$$
		Then  $||A||_{\text{op}}\leq 1$, as taking submatrices does not increase the operator norm. Using $A$ as our test matrix, we thus get
		$$||M||_{\tr}\geq \langle M,A\rangle=q^{-\frac{dD}{2}}\sum_{a,x\in Q} \mathds{1}_{\langle a,x\rangle=0}\cdot \chi(\langle a,x\rangle).$$
		The sum on the right-hand-side is exactly the number of one-entries of $M$, so by \Cref{lemma:num_one_entries}, we get
		$$||M||_{\tr}\geq q^{-\frac{dD}{2}}n^{2-\frac{2}{d}-\frac{c}{2}}\geq n^{\frac{3}{2}-\frac{2}{d}-c}.$$
		Here, in the third inequality, we used that $q\leq 3R\leq 3\sqrt{C}n^{\frac{1}{dD}}$, so $q^{dD/2}\leq 3^{dD/2}C^{dD/4}n^{1/2}\leq n^{1/2+c/2}$, assuming $R$ is sufficiently large with respect to $C,c,d$. This finishes the proof by the inequality $\gamma_2(M)\geq\frac{1}{n}||M||_{\tr}$.
	\end{proof}
	
	Finally, we prove that $M$ has constant randomized communication cost. This argument is essentially identical to the one presented by Goh and Hatami \cite{GH_totally_real}, so we only present a brief sketch.
	
	\begin{lemma}
		$R(M)=O(d(\log d+\log R)).$
	\end{lemma}
	
	\begin{proof}
		Let $L=dR^2$. For $y\in B_K(L)$, define $J(y)\subset [D]$ to be the set of indices $i$ such that $|\sigma_i(y)|\leq L^{-3}$. Then by \Cref{lemma:norm}, we have
		$$1\leq |N(y)|=\prod_{i=1}^D |\sigma_i(y)|\leq L^D L^{-3|J(y)|}.$$
		Therefore, $|J(y)|\leq D/3$.
		
		Now let $(a,x)\in Q^2$, then Alice and Bob need to guess whether $\langle a,x\rangle=0$, while Alice only has access to $a$, and Bob only has access to $x$. The shared random string allows them to agree in a randomly selected index $i$  from the uniform distribution on $[D]$. Alice computes the evaluations $\sigma_i(a_1),\dots,\sigma_i(a_d)$ of the coordinates, and she chooses real number $z_1,\dots,z_d$ such that $|\sigma_i(a_j)-z_j|\leq \frac{1}{3dL^3R}$ for $j\in [d]$. As $|\sigma_i(a_j)|\leq R$, she can choose such a $z_j$ whose binary representation has $O(\log d+\log R)$ bits. Alice then shares the tuple $(z_1,\dots,z_d)$ with Bob, costing $O(d(\log d+\log R))$ bits of communication.
		
		Next, Bob evaluates 
		$$E:=\sum_{j=1}^d z_j \sigma_i(x_j).$$
		Finally, Bob guesses $\langle a,x\rangle=0$ if and only if $|E|\leq \frac{1}{2L^3}$. We show that Bob guesses correctly with probability at least $2/3$.
		
		We have
		$$|E-\sigma_i(\langle a,x\rangle)|=\left|\sum_{j=1}^d (z_j-\sigma(a_j))\sigma(x_j)\right|\leq d\cdot \frac{1}{3dL^3R}\cdot R=\frac{1}{3L^3}.$$
		Therefore, if $\langle a,x\rangle=0$, then $|E|\leq 1/(3L^3)$, so Bob guessed correctly. Assume that $y:=\langle a,x\rangle\neq 0$. Then $y\in B_K(L)$ and $|J(y)|\leq D/3$. Therefore, with probability at least $2/3$, the shared index $i$ does not fall into $J(y)$. But if $i\in [D]\setminus J(y)$, then  $|\sigma_i(y)|\geq \frac{1}{L^3}$, so $|E|\geq \frac{2}{3L^3}$ as well. Therefore, in this case, Bob guessed correctly again. In conclusion, with probability at least $2/3$, Bob guessed correctly whether $\langle a,x\rangle=0$.
	\end{proof}

    \section*{Acknowledgments}

    The authors would like to thank Nikhil Bansal and Haotian Jiang for useful discussions on discrepancy, and also Lianna Hambardzumyan for discussions on factorization norms. We would also like to thank Jakob Hultgren and Muhammad Sohaib Khalid for helping with algebraic geometry.

    \section*{Declaration on the use of AI}

    ChatGPT 5.5 and 5.6-Sol was used in the preparation of the paper. These models played a significant role in extending the proof of \Cref{thm:upper_hyperplane} to the more general setting of \Cref{thm:variety}, and in the proof of \Cref{lemma:character}. All AI-generated suggestions were reviewed by the authors, who take full responsibility for the correctness and originality of the paper.

\end{document}